\documentclass[preprint]{elsarticle}

\usepackage{lineno,hyperref}
\modulolinenumbers[5]

\usepackage{amsthm}
\usepackage{amsmath}
\usepackage{amssymb,textgreek}
\usepackage{thmtools, thm-restate}

\usepackage{enumerate}
\usepackage{color,xcolor}
\usepackage{tikz}
\usetikzlibrary{automata}
\usetikzlibrary{arrows,shapes,calc}
\usetikzlibrary{positioning}
\usepackage{soul}

\journal{Stochastic Processes and Application}

\usepackage{amssymb}
\usepackage{amsthm}
\usepackage{mathtools}

\def\gr{{\rm GR}}
\def\sw{{\rm SW}}
\def\aw{{\rm AW}}

\def\bZ{{\mathbb Z}}

\def\bP{{\mathbb P}}
\def\bE{{\mathbb E}}

\def\cX{\mathcal X}

\def\bfeta{\boldsymbol{\eta}}
\def\bfX{\boldsymbol{X}}

\def\bfQ{\boldsymbol{Q}}
\def\bfY{\boldsymbol{Y}}

\newcommand\1{\leavevmode\hbox{\rm \small1\kern-0.35em\normalsize1}}

\def\mExp{\mbox{Exp}}
\def\bR{\mathbb{R}}
\def\N{\mathbb N}

\newtheorem{theorem}{Theorem}
\newtheorem*{terminology}{Terminology}

\newtheorem{remark}{Remark}

\newtheorem{lemma}{Lemma}
\newtheorem{assumption}{Assumption}

\begin{document}

\begin{frontmatter}

\title{Self-Switching Markov Chains: emerging dominance phenomena}

\author[address1]{S. Gallo}
\ead{sandro.gallo@ufscar.br}

\author[address2]{G. Iacobelli}
\ead{giulio@im.ufrj.br}

\author[address2]{G. Ost }
\ead{guilhermeost@im.ufrj.br}
\cortext[mycorrespondingauthor]{Corresponding author}

\author[address3]{D. Y. Takahashi}
\ead{takahashiyd@gmail.com}

\address[address1]{Universidade de São Carlos, São Paulo, Brazil}
\address[address2]{Universidade Federal do Rio de Janeiro, Rio de Janeiro, Brazil}
\address[address3]{Universidade Federal do Rio Grande do Norte, Natal, Brazil}

\begin{abstract}
The law of many dynamical systems changes with the evolution of the system. These changes are often associated with the occurrence of certain events whose time of occurrence depends on the trajectory of the system itself. Dynamics that take longer to change will be observed more frequently and may dominate in the long run (the only ones observed).  This article proposes a Markov chain model, called Self-Switching Markov Chain, in which the emergence of dominant dynamics can be rigorously addressed. We present conditions and scaling under which we observe with probability one only the subset of dominant dynamics. 

\end{abstract}


\end{frontmatter}

\tableofcontents

\section{Introduction}
In many dynamical systems observed in Nature, the law of the dynamics changes along with the systems' temporal evolution. These changes are usually associated with the occurrence of certain events. For example, during foraging, an animal searches for food until coming across a predator. If the animal successfully escapes,  it will resume the search blue soon after. As another example, a robot explores an environment until it hits a wall, after which it continues the exploration by following some rule. As a third example, animals exploit a behavior (e.g., sleeping, eating) until an event happens (e.g., body temperature increases, becomes satiated), making them explore another behavior (e.g., wakes up, takes a nap). All these instances have in common that the timing of a change in the dynamics depends on its trajectory. Naturally, trajectories that take longer to satisfy the event (e.g., meeting a predator) will last longer. Therefore, we expect to observe more frequently those dynamics that take longer to change in the long run. The consequence is that, in general, few dynamical states might dominate our observations of natural phenomena. This article proposes a simple Markov chain model that modifies the dynamics depending on its trajectory, similar to the examples described above. We call the model Self-Switching Markov Chains (SSMC). Under general conditions and scaling, we prove that we observe with probability one only a subset of the dynamics, and  we characterize these dominant dynamics. Furthermore, we show that the switching between dynamics exhibits metastability-like properties. 
   
The  phenomenon of dominant dynamics emergence exhibits a mechanism that is akin to natural selection in the theory of evolution.  In evolution, the phenotypes that we consider more adapted are the ones that last more. The difference between natural selection and the dominance phenomenon that we describe is that in evolution, the selection happens in a population ensemble and not at the individual trajectories. The precise formulation of the idea that selection at the individual level can explain the predominance of certain animal behaviors seems to be new. The standard approach for the predominance of a subset of actions in an individual is to assume a learning mechanism \cite{slater1999essentials}. Our result suggests that learning is not the only mechanism that modulates the preference for one dynamics over the other.  A behavior can be dominant simply because it takes longer to change.  

Our model is related to some other popular models in the literature. Hidden Markov Models (HMM) have been widely used to model switching dynamics like animal behavior. HMM is similar to SSMC because the dynamics change depending on the state of the hidden process, but in HMM the change of hidden state is independent of the trajectory of the observed process \cite{cappe2006inference}. Therefore, we do not observe the dominance phenomenon in HMM. Markov decision processes select actions based on its current sate. It also models the predominance of specific behaviors but uses the idea of reward learning mechanism \cite{puterman2014markov}. Finally, Hybrid Switching Diffusions are closely related to SSMC, but there is no description of emergent dominant phenomena for this class of processes to the best of our knowledge \cite{yin2009hybrid}. 

The article is organized as follows. In Section~\ref{sec:model} we introduce the model of Self-Switching Markov Chains (SSMC) and discuss some of its properties. In Section~\ref{sec:Ntoinfty} we study the asymptotics of sequences of SSMC and provide the main results about emergence of dominant behaviors.
Section~\ref{sec:examples} provides a few concrete examples of SSMC based on random walks models.   Finally,  Section~\ref{sec:final} concludes the paper with a brief discussion about  possible extensions of this work.
The proofs of the  technical results are postponed to \ref{app:proofs}, while in  \ref{app:lemmas}  the proofs of couple of auxiliary lemmas are given. 



\section{Model and main results} 

\paragraph{General notation.}
Let  $\N:=\{0,1,\ldots\}$ be the set of natural numbers. Henceforth,  we shall use the following notation: 
$\mathcal X$ denotes a finite set (describing the possible configuration of a system) and $\Xi$  a compact metric space (set of possible states). 

\subsection{Self-Switching Markov Chains}\label{sec:model} A \emph{Self-Switching Markov Chain} (SSMC) is a Markov chain  $\{(X_i, \eta_i)\}_{i\in \mathbb{N}}$ taking values on $\mathcal X\times\Xi$ which   depends on four parameters:
\begin{itemize}
    \item $\mu$, a probability measure on $\Xi$ (\emph{state distribution});
    \item $\bfQ:=(Q^{(\theta)})_{\theta\in\Xi}$,  a family  of $\mathcal X\times\mathcal X$ stochastic matrices (\emph{dynamic in state $\theta$});
    \item $x_0\in \mathcal X$ (\emph{origin});
    \item $T\subset\mathcal X$,  a set which does not contain the origin (\emph{target set}).
\end{itemize}
%
The dynamics of a SSMC  is defined as follows:  
 $(X_0,\eta_0)\sim\delta_{x_0}\otimes\mu$,  and for all $i\ge1$, all $x_{<i}:=(x_0,\ldots,x_{i-1})\in\mathcal{X}^{i}$ and all $\theta_{<i}:=(\theta_0,\ldots,\theta_{i-1})\in\Xi^i$
\begin{align*}
(X_{i},\eta_i)|X_{<i}=x_{<i},\eta_{<i}=\theta_{<i}\sim \begin{cases}
\delta_{x_0}\otimes\mu,  & \text{ if  $x_{i-1}\in T$}\;,
    \\
Q^{(\theta_{i-1})}(x_{i-1},\cdot)\otimes\delta_{\theta_{i-1}}, & \text{otherwise}\;.
    \end{cases}
\end{align*}
We denote by $\mathbb{P}_{x_0,\mu}$ the law of $\{(X_i, \eta_i)\}_{i\in \mathbb{N}}$  when $(X_0, \eta_0)\sim \delta_{x_0}\otimes\mu$, and by $\mathbb{E}_{x_0,\mu}$ the corresponding expectation, while the expectation with respect to $\mu$ is denoted by $\mathbf{E}_\mu$. 
The first coordinate process  $\bfX=\{X_i\}_{i\in \mathbb{N}}$  may be seen as a \emph{location} process, while the second coordinate  $\bfeta=\{\eta_i\}_{i\in \mathbb{N}}$  as a  process determining the \emph{state} of the process $\bfX$. The location and the state process are coupled;  on one hand, at time $i$, the location $X_i$ is updated through a transition matrix which is  parametrized by $\eta_i$, while the   process $\bfeta$ is  updated as soon as  the  location process $\bfX$ hits the target set  $T\subset\mathcal X$. After each such event, the location process restarts from the origin  $x_0$ and a new state is sampled according to $\mu$ and independently of everything else. In other words, the random times when the location process hits the target set are renewal times for the location process.

\paragraph{Example}  As an example of a SSMC consider a random walk $\bfX$ on a weighted finite (connected) graph $G=(V,E, w)$, where $w:E\to [a,b]$, with $b>a>0$,  is a function assigning to every edge $e \in E$ a weight $w(e) \in [a,b]$.
In this case  $\mathcal{X}=V$, $\Xi=[a,b]^E$ and $Q^{(\theta)}$, with $\theta \in \Xi$,  is the transition matrix of a random walk on $G$ with edge  weights $w=\theta$.
The edge weights correspond to a particular state and $\eta_i$ denotes the weights used at time $i$. 
Given an origin $x_0 \in V$, and a set of target vertices $T\subset V$, the dynamics of this SSMC consists in drawing the edge weights according to a distribution $\mu$ and  letting  the random walk move on the corresponding  weighted graph, starting from $x_0$. As soon as the random walk hits the target set $T$, a new weight function is sampled according to $\mu$ and the random walk restarts from the origin. 

\vspace{0.3cm}

Throughout the paper, we assume that $Q^{(\theta)}$ is irreducible for any $\theta\in\Xi$. This guarantees that the process $\bfX$  will visit $T$ infinitely often, since $\mathcal X$ is finite. 
However, this does not automatically grant positive recurrence of $\bfX$, due to the coupled dynamics with the state process $\bfeta$. 

To better understand the coupled dynamics, let us define the following sequence of stopping times: $S_0:=0$ and for any $i\ge1$
\[
S_i:=\inf\{n>S_{i-1}:X_n\in T\}\;.
\]
Note that the distribution of the stopping time $S_i$ depends on the whole past of the chain, and in particular depends on the processes $\bfeta$ up to time $S_{i-1}+1$.  However,  conditioned on $S_{i-1}$, the distribution of $S_i$ depends on $\bfeta$ only through $\eta_{S_{i-1} +1}$. By the definition of the model, $\eta_{S_{i-1} +1}$ is distributed according to $\mu$ and is independent of everything else. Thus, if for every $i\geq 1$,  we define 
the \emph{inter-switching} times $\tau_i:=S_{i}-S_{i-1}$, and denote by 
$\{\Theta_i\}_{i\geq 1}$  an i.i.d. sequence of random variables with distribution $\mu$ (independent of everything else), we have that $\tau_i(\eta_{S_{i-1} +1}) \sim \tau_i(\Theta_i)$, for every $i\geq 1$.
\begin{remark}\label{rem:iid}
The  random variables $\{\tau_i(\Theta_i)\}_{i\geq 1}$ are i.i.d. However, 
$\mathbb{P}_{x_0,\mu}(\tau_i(\Theta_i)=\cdot | \Theta_i=\theta)$ strongly depends on $\theta \in \Xi$, i.e., different states give rise to different inter-switching times. 
\end{remark}

Given the i.i.d. nature of the sequence $\{\tau_i(\Theta_i)\}_{i\geq 1}$,  we shall denote by $\tau_1(\Theta_1)$ anyone of them. We shall also use the notation $\tau(\Theta)$ or simply $\tau$, when there is no risk of confusion. 
Throughout the paper we shall {work under the following assumption}.
\begin{assumption}
\label{Ass:1}
The function 
\begin{equation}\label{eq:def_m}
\Xi \ni \theta\mapsto m(\theta):=\bE_{x_0, \mu}\left[\tau_1(\Theta_1)|\Theta_1=\theta\right]\;,
\end{equation} is continuous. 
\end{assumption}

Assumption~\ref{Ass:1} guarantees that the process $\bfX$ is positive recurrent, i.e., the   time elapsed between consecutive visits to $T$ has finite expectation. In fact,  since $Q^{(\theta)}$ is irreducible for every $\theta \in \Xi$ and $\mathcal X$ is finite, we know that $m(\theta)<\infty$. Moreover, since  $\Xi$ is compact, the function $m$ is uniformly bounded and thus
\begin{equation}
\label{def:normalization_contant}
\bE_{x_0,\mu}\left[\tau_1(\Theta_1)\right]=\mathbf{E}_{\mu}[m]=\int_{\Xi}m(\theta)\mu(d\theta)<\infty\;.
\end{equation}

\paragraph{Occupation measure of the state process}

For each integer $n\geq 1$ and Borelian set $A\subseteq \Xi$, let us denote by $\hat{p}_{n}(A)$ the \emph{empirical occupation measure} of $A$ for the process $\bfeta$, defined as 
\begin{equation}
\label{def:P_hat_N_n}
\hat{p}_{n}(A):=\frac{1}{n}\displaystyle \sum_{k=1}^n \1_{(\eta_k\in A)}\;.
\end{equation}

The following proposition provides the limiting measure, as time $n$ goes to infinity, of the  empirical occupation measure {\color{blue} (Proof in \ref{sec:proofs_as_distribution}).}   

\begin{restatable}{proposition}{propLLN}
\label{Prop:P_hat_N_n_convto_P_hat_N}
Let $m:\Xi \to \mathbb{R}$ be the function defined in \eqref{eq:def_m}. Under Assumption \ref{Ass:1}, for every Borel measurable $A\subseteq \Xi$
it holds that  
\begin{equation*}
\hat p_{n}(A)\underset{n\to \infty}{\longrightarrow}\; 
\widehat{P}(A):=
\frac{\mathbf{E}_\mu\left[\1_A m\right] }{\mathbf{E}_\mu[m]}\;, \quad \text{$\bP_{x_0, \mu}$-almost surely}\;.
\end{equation*}
\end{restatable}


Note that the limiting probability measure $\widehat{P}$ on $\Xi$ is absolutely continuous with respect to the probability measure $\mu$ {since}
$$
\widehat{P}(d\theta)=\frac{m(\theta)}{\mathbf{E}_\mu[m]}\mu(d\theta)\;.
$$

\subsection{Sequence of Self-switching Markov chains}\label{sec:Ntoinfty}

Recall that a SSMC  $\bfY=\{(X_i, \eta_i)\}_{i\in \mathbb{N}}$ is defined by the pair $(\Xi,\mu)$, and  the tuple  $(\bfQ, \mathcal X,x_0,T)$. 
In this section, we will consider a sequence $\bfY_N,N\ge1$ of SSMCs  using a sequence of tuples  $(\bfQ_N, \mathcal X_N,x_0, T_N),N\ge1$ in which $\mathcal X_N,N\ge1$ is  an increasing sequence of finite sets, $x_0 \in \mathcal{X}_1$,  $T_N\subset\mathcal X_N$ are subsets which do not contain $x_0$, and finally, $\bfQ_N,N\ge1$ are irreducible stochastic matrices on $\mathcal X_N$. With the exception of $\Xi$, $\mu$ and $x_0$, all the quantities which were introduced in Section \ref{sec:model} are affected by the dependence on $N$, and for this reason, they will be indexed by $N$. In particular, $\widehat P$, $\{\tau_i\}_{i\geq 1}$,   $\tau$ and $m(\theta)$ will now be denoted respectively as $\widehat P_N$, $\{\tau_{i,N}\}_{i\geq 1}$,  $\tau_N$  and $m_N(\theta)$ for any $N\ge1$.
The sets $T_N$ and the parameter $N\ge1$ control the size of the configuration space of the location process. 


\subsubsection{Emerging dominance}

In this section, we  investigate  the possible limits of the sequence of probability measures $\{\widehat{P}_N\}_{N\geq 1}$  where we recall that, for any $N\ge1$, $\widehat{P}_N(A):=\frac{\mathbf{E}_\mu\left[\1_A m_N\right] }{\mathbf{E}_\mu[m_N]}$. 
We will present general conditions   under which 
the  sequence $\{\widehat{P}_N\}_{N\geq 1}$ converges weakly to a mixture of  Dirac deltas. 

\medskip 
In what follows, we write ${\rm supp}(\mu)$ to indicate the support of the probability measure $\mu$. 

\begin{terminology}[Emerging dominance]

 Given $\{\theta^*_1,\ldots, \theta^*_K\} \subsetneq {\rm supp}(\mu)$ (proper inclusion), we say that there  is ``\emph{emerging dominance at $\{\theta^*_1,\ldots, \theta^*_K\}$ with positive weights $\{w_i,\ldots,w_K\}$}" if the sequence of probability measures $\{\widehat{P}_N\}_{N\geq 1}$ converges weakly, when $N$ diverges, to the mixture of Dirac deltas $\sum_{i=1}^Kw_i\delta_{\theta^*_i}$, and $\sum_{i=1}^Kw_i=1$.

\end{terminology}

We emphasize that emerging dominance requires that the sequence $\{\widehat{P}_N\}_{N\geq 1}$ converges weakly to a mixture of Dirac deltas supported on a {\em proper subset} of ${\rm supp(\mu)}$. 
In light of that, note that, if $\limsup_{N\to \infty}m_N(\theta)\leq L$, $\forall \theta \in \Xi$, then $\liminf_{N\to \infty}\widehat{P}_N(A):=\liminf_{N\to \infty}\frac{\mathbf{E}_\mu\left[\1_A m_N\right] }{\mathbf{E}_\mu[m_N]}\geq \frac{\mu(A)}{L}$. Hence,  the support of any weak limit of $\{\widehat{P}_N\}_{N\geq 1}$ coincides with ${\rm supp(\mu)}$, and therefore there is no emerging dominance. Thus, henceforth, we consider only situations in which $m_N(\theta)$ diverges, as $N$ tends to infinity.

\medskip 

 The simplest situation is when ${\rm supp}(\mu)$ is a finite set. In this case we have the following theorem  (Proof in \ref{sec:proofs_as_distribution}). 
%

\begin{restatable}{theorem}{thmFinite}
\label{theo:discrete}
Let $(\Xi,d)$ be a compact metric space and $\mu$ a discrete  probability measure on $\Xi$ with finite support. 
Assume that there exists $\{\theta^*_1,\ldots, \theta^*_K\}\subsetneq  {{\rm supp}(\mu)}$
satisfying  
\begin{enumerate}[(1)]
    \item there exists  $1\leq i\leq K$ such that for all $\theta \in {\rm supp}(\mu)\setminus\{\theta^*_1,\ldots, \theta^*_K\}$, $$\frac{m_N(\theta)}{m_N(\theta^*_i)} \underset{N\to \infty}{\longrightarrow} 0\;,$$
 \item  for each $1\leq i,j\leq K$ with $i\neq j$, it holds 
 $$
 \frac{m_N(\theta^*_j)}{m_N(\theta^*_i)} \underset{N\to \infty}{\longrightarrow} d_{i,j}>0\;.
 $$
\end{enumerate}
Then, we have emerging dominance at $\{\theta^*_1,\ldots, \theta^*_K\}$ with weights
\[
w_{i}= \frac{1}{1 + \sum_{\substack{j=1\\ j \neq i}}^K \frac{\mu(\theta^*_j)}{\mu(\theta^*_i)}d_{i,j}}\,,\,\,i=1,\ldots,K.
\]
\end{restatable}

 {This result solves the question in the finite case in a rather satisfactory way. Indeed, the conditions are quite natural. Condition (1) states that the time spent in any state is asymptotically negligible compared to the $K$ dominating states, while on the other hand Condition (2) asks that the asymptotic time spent in the dominating states be comparable. Condition (1) cannot be relaxed in general as will be explained in Remark \ref{rem:opt}. Condition (2), as far as the convergence holds, cannot be relaxed in general, since if for some pair $(\theta^*_i,\theta^*_j)$ the ratio diverges or vanishes, this means that one the two is asymptotically negligible and should be removed from the set of candidates to dominant states in first place. }
 
 \medskip

The second result relies on assuming some regularity on the sequence of  functions $h_N(\theta)=N^{-1}\log(m_N(\theta))$, and it  provides sufficient conditions for the emergence of a {\em unique} dominant state {\color{blue} (Proof in \ref{sec:proofs_as_distribution}).}   

\begin{restatable}{theorem}{thmUnique}
\label{theo:unique}
Assume that the function $\Xi\ni \theta\mapsto h_N(\theta)=N^{-1}\log(m_N(\theta))$ converges uniformly to a function $h$ as $N\to\infty$. If there exists a unique parameter $\theta^*\in {\rm supp}(\mu)$  such that    
$\arg\max_{\theta\in\Xi} h(\theta)=\theta^*$, and  $h(\theta^*)>0$,
then we have emerging dominance at $\{\theta^*\}$, with weight $1$.
\end{restatable}

\begin{remark}\label{rem:laplace1}
Theorem~\ref{theo:unique}  may also be formulated in a more general setting  {which relates the problem to the so-called Laplace's Principle}.

Let $\{g_N\}_{N\geq 1}$ be a sequence of functions defined on a compact metric space $(\Xi,d)$ which converges uniformly to a function $g$. Let $\mu$ be a probability measure on $\Xi$. If  $\theta^*=\arg\max_{\theta\in \Xi}g(\theta)$ (that is, $g$ has a unique global maximum at $\theta^*$), $g$ is continuous at $\theta
^*$,
and $\theta^*\in {\rm supp}(\mu)$, then the following generalized Laplace's Principle holds: for any continuous function $f:\Xi\to\bR$
$$
\frac{\int_{\Xi}f(\theta)e^{Ng_N(\theta)}\mu(d\theta)}{\int_{\Xi}e^{Ng_N(\theta)}\mu(d\theta)}\;\underset{N\to\infty}{\longrightarrow}\; f(\theta^*)\;. 
$$

\end{remark}

  {One advantage of Theorem \ref{theo:unique} is that it dispenses any condition on the probability measure $\mu$ and its support. The drawback is that it gives emerging dominance only for  a single state. However}, from the analysis of the  sequence of functions $h_N(\theta)=N^{-1}\log(m_N(\theta))$,  {and assuming that $\mu$ is absolutely continuous w.r.t. the Lebesgue measure} it is also possible to infer the emergence of more than one dominant state. Our   next theorem provides sufficient conditions under which  $\{\widehat{P}_N\}_{N\geq 1}$  converges to a combination of two Dirac deltas in the one-dimensional case.

 
\begin{restatable}{theorem}{thmdouble}\label{theo:twodominance}
Let $\Xi=[a,b]\subset\bR$ with  $a<b$  and assume that the probability measure $\mu$ on $\Xi$ has continuous density  $g_{\mu}$ with respect to the Lebesgue measure. Assume moreover that
 \begin{enumerate}
     \item the sequence $h_N$ converges uniformly to a function $h\in C^1([a,b])$,
     \item the function $h$ satisfies $\arg\max_{\theta\in\Xi} h(\theta)\cap \,{\rm supp}(\mu)=\{\theta^*_1,\theta^*_2\}=\{a,b\}$ with $h(a)=h(b)>0$, $h'(a)<0$ and $h'(b)>0$\footnote{With an abuse of notation, we are denoting by $h'(a)$ the right derivative at $a$, and  by $h'(b)$ the left derivative at $b$. },
     \item for each $1\leq i\leq 2$, it holds that  
\[N\left(h_N(\theta^*_i)-h(\theta^*_i)\right)\to d_i\;, \text{ with } |d_i|<\infty\;,
\]
and
\[
 \sup_{\theta\in [\theta^*_i-\delta,\theta^*_i+\delta]\cap [a,b]}|h'_N(\theta)-h'(\theta)|\to0\;,
\]
{for all $\delta>0$ sufficiently small}.
     \end{enumerate}
Then there is emerging dominance at $\{\theta_1^*,\theta_2^*\}$ with weights  $\{w_1,w_2\}$ given  by
\[
w_i=\frac{(-1)^{i}e^{d_i}g_{\mu}(\theta^*_i)/h'(\theta^*_i)}{e^{d_2}g_{\mu}(\theta^*_2)/h'(\theta^*_2)-e^{d_1}g_{\mu}(\theta^*_1)/h'(\theta^*_1)}\;, \quad \ i=1,2\;.
\]
\end{restatable}

The third condition is   slightly technical, on the convergence of $h$ and $h'$ at (or around) $a$ and $b$. Observe that Assumption 1 was already present in Theorem \ref{theo:unique}, and that Assumption 2 is rather natural, telling us that behaviors close to $a$ and $b$ have negligible expectation compared to behaviors $a$ and $b$, which have comparable (equal) expected time.


\begin{restatable}{remark}{remdouble}
\label{cor:1d_K2}  
Theorem \ref{theo:twodominance} deals with the case in which the dominant states are at the boundary of the state space $\Xi$. One can also deal with the case in which the dominant states are in the interior of $\Xi$. This can be obtained under conditions slightly different than those of Theorem \ref{theo:twodominance}.   
Still with $\Xi=[a,b]$, $a,b\in\bR$ with $a<b$, and assuming that the probability measure $\mu$ on $\Xi$ has continuous density  $g_{\mu}$ with respect to the Lebesgue measure. Assume moreover that the functions $h_N(\theta)=N^{-1}\log(m_N(\theta))\in C^2([a,b])$ satisfy the following conditions:
\begin{enumerate}
    \item $h_N$ converges uniformly to a function $h\in C^2([a,b])$. 
    \item The function $h$ satisfies $\arg\max_{\theta\in\Xi} h(\theta)\cap {\rm supp}(\mu)=\{\theta^*_1,\theta^*_2\}$, where $\theta^*_1,\theta^*_2\in (a,b)$ are such that $\theta^*_1\neq \theta^*_2$,   $h(\theta^*_1)=h(\theta^*_2)>0$, $h'(\theta^*_1)=h'(\theta^*_2)=0$ and $h''(\theta^*_1)=h''(\theta^*_2)<0.$
    \item Moreover, for each $1\leq i\leq 2$, it holds that $N(h_N(\theta^*_i)-h(\theta^*_i))\to d_i$ where $|d_i|<\infty$, $|h'_N(\theta^*_i)|\leq C/N$ for all $N$ sufficiently large and for all $\delta>0$ sufficiently small $\sup_{\theta\in [\theta^*_i-\delta,\theta^*_i+\delta]}|h''_N(\theta)-h''(\theta)|$ converges to $0$ as $N\to\infty$.
\end{enumerate}
Then there is emerging dominance at $\{\theta_1^*,\theta_2^*\}$ with weights  $\{w_1,w_2\}$ given by ,
$$
w_i=\frac{e^{d_i}g_{\mu}(\theta^*_i)/\sqrt{-h''(\theta^*_i)}}{e^{d_1}g_{\mu}(\theta^*_1)/\sqrt{-h''(\theta^*_1)}+e^{d_2}g_{\mu}(\theta^*_2)/\sqrt{-h''(\theta^*_2)}}\;, \quad \ i=1,2\;.
$$
\end{restatable}


Theorem  \ref{theo:twodominance} and Remark \ref{cor:1d_K2} are proven in \ref{sec:proofs_as_distribution}. Their proof rely on Laplace's method for asymptotic evaluation of integrals which has been  extensively used in statistical theory, specially in Bayesian framework, to approximate integrals with respect to posterior distributions (see for instance \cite{Kass_et_al_99} and \cite{Wong1989Asymptotic}).

Our last theorem of the section complements the previous results by providing sufficient conditions under which there is no emergence of dominant state, i.e., the sequence of probability measures $\{\widehat{P}_N\}_{N\geq 1}$ converges weakly to a probability measure on $\Xi$ which is not a mixture of Dirac deltas. 

\begin{theorem}\label{th:nodominance}
If there exists a sequence  $(a_N)_{N\geq 1}$ such that $m_N/a_N$ converges monotonically $\mu$-almost surely to a positive function $\overline{m}$  such that $\mathbf{E}_\mu\left[\overline{m} \right]<\infty$, then $\widehat{P}_N(A) \to \frac{\mathbf{E}_\mu\left[\1_A \overline{m}\right] }{\mathbf{E}_\mu[\overline{m}]} $.
\end{theorem}

The proof of Theorem~\ref{th:nodominance} follows directly from  Monotone Convergence Theorem.

\subsubsection{Metastability}
 
We say that a state $\theta$ is \emph{metastable} if the inter-switching time $\tau_N$ corresponding to $\theta$ is asymptotically exponentially distributed, that is,
\[
\bP_{x_0, \delta_\theta}\left(\frac{\tau_{N}}{\bE_{x_0,\delta_\theta}[\tau_{N}]}>t\right) \;\underset{N\to\infty}{\longrightarrow}\; e^{-t}\;.
\]
Here, $\bP_{x_0, \delta_\theta}$ is the probability corresponding to a SSMC with $\mu=\delta_\theta$.   Observe that ${\bE_{x_0,\delta_\theta}[\tau_{N}]}= \bE_{x_0, \mu}\left[\tau_N(\Theta_1)|\Theta_1=\theta\right]$ was denoted by $m_N(\theta)$ in the previous subsections.
This terminology comes from the fact that exponential distribution of escape times from potential wells is the main feature of metastability \cite{fernandez2015asymptotically}.


Conversely, we say that a state $\theta \in \Xi$ exhibits \emph{cut-off} (following \cite{bertoncini2008cut}), if  there exist positive constants $c_1$ and $c_2$ (independent of $N$), such that
\[
\bP_{x_0, \delta_\theta}\left(c_1< \frac{\tau_{N}}{\bE_{x_0,\delta_\theta}[\tau_{N}]}<c_2\right) \;\underset{N\to\infty}{\longrightarrow}\; 1\;.
\]

Naturally, metastability and cut-off are two extreme situations, and many other situations can occur in between. However, the duality metastability \emph{vs.} cut-off has an interesting interpretation. The former corresponds to states that last an unpredictable random time -- the switching time cannot be predicted, while the later corresponds to  states that last for deterministic time. For SSMC metastability and cut-off may alternate in a single path as we will illustrate in the next section by simple examples. 

%
%

In order to establish metastability, we will make use of the following theorem, due to  \cite{fernandez2015asymptotically}. 

\begin{theorem}[Theorem~2.7 in \cite{fernandez2015asymptotically}]\label{thm:fernandezetal}
Let $\tau_N$ be the inter-switching time and denote by $\tau_{x_0}$ the first time the location process $\bfX$ visits the origin $x_0$. Suppose there exist two sequences of real numbers $(R_N)_{N\ge1}$ and $(r_N)_{N\ge1}$ satisfying $R_N>0$, $\frac{R_N}{\bE_{x_0, \delta_\theta}[\tau_N]}\rightarrow0$, $r_N\in(0,1)$, $r_N\rightarrow0$ such that 
\[
\sup_{x\in\mathcal X_N}\bP_{x,\delta_\theta}(\min\{\tau_N, \tau_{x_0}\}>R_N)\le r_N\;, \quad  N\ge1\;.
\] 
Then, $\frac{\tau_N}{\bE_{x_0,\delta_\theta}[\tau_N]}$ converges in distribution (with respect to $\bP_{x_0,\delta_\theta}$) to an exponential random variable with parameter $1$;  thus the state $\theta$ is metastable. 
\end{theorem}

\section{Self-switching environment Random Walks}\label{sec:examples}

The SSMCs were introduced in a rather general way, without detailing the dynamic of the location process, encoded by the matrices $Q^{(\theta)},\theta\in\Xi$. We now consider specific cases to illustrate  the phenomena discussed in Section~\ref{sec:Ntoinfty}.  These examples are  built upon one-dimensional simple random walks models on increasing subsets of $\mathbb Z$.
In all the examples, we assume that $\Xi=[0,1]$ and that $\mu$  has support in a compact subset of $(0,1]$.

 \subsection{Physical interpretation}
 Given a transition matrix $Q$ of nearest neighbors random walk  on $\cX\subset \bZ$ (containing $0$), we call ``potential'' the  function $V:\mathbb R\rightarrow \mathbb R$ interpolating the following values
\begin{equation}\label{eq:potential} 
V(n)=\left\{
\begin{array}{cc}
\sum_{i=1}^{n}\log \frac{Q(i,i-1)}{Q(i,i+1)}&n\in\{1,2,\ldots\}\cap\cX\;,\\
0&n=0\\
-\sum_{i=n+1}^{0}\log \frac{Q(i,i-1)}{Q(i,i+1)}&n\in\{-1,-2,\ldots\}\cap\cX\;.
\end{array}
\right.
\end{equation}
The potential gives us a physical intuition of the different dynamics. We refer to Figure \ref{fig:potential} for a schematic illustration of the three types of potentials we are considering below.

 \subsection{Biological interpretation}
 Animal behaviors are often thought to be random walks in a state space with transition probabilities controlled by few parameters \cite{codling2008random}. For example, when observing the movement of a fly inside a chamber, we think of the sequence of positions of the fly as following a random walk in constrained state space. As another example, we can use random walks with barriers to model the neural process of making decisions \cite{fific2010logical}. In these models, when a random walk reaches one of the barriers in the state space, it means that the person made a decision. Different behaviors of the animal are modeled by distinct transition probabilities of the random walk. In both of our examples, whenever the random walk reaches the barrier, we can think that it restarts the dynamics following transition probabilities that are chosen dynamically. The SSERW introduced below are simple one-dimensional models that capture these ideas. The key phenomenon is the emergence of dominating states (transition probabilities) even when the states are chosen independently of the process. In the fly example, that would imply that, in the long run, we would expect few flying strategies to dominate the observations even without a learning mechanism. Similarly, for the decision-making, we expect that, in the long run, we primarily observe few decision-making strategies, despite the lack of learning. The dominance of few states is in agreement with the observation that only few major behavioral states are observed in animals \cite{brown2018ethology}.

\subsection{Flat potential SSERW}\label{sec:gr}

Let $\cX_N=[-N,N] \cap \bZ$, $T_N=\{-N,N\}$ and $x_0=0$.  
For any $N\ge1$, let  $\bfQ_N={\bf Q}$, where  for any $\theta\in\Xi$, $Q^{(\theta)}$ is defined as follows: for  $x \in \cX_N\setminus \{-N,N\}$
\begin{align}
\label{def:gambler-ruin}
Q^{(\theta)}(x,y)=
\begin{cases}
\theta\;, & \text{if} \ y=x+1\;, \\
1-\theta\;, & \text{if} \ y=x-1\;, \\
0\;, & \text{otherwise}\;,
\end{cases}
\end{align}
whereas $Q_N^{(\theta)}(x,\cdot)$ can be chosen  arbitrary for $x \in \{-N, N\}$ because it does not influence the dynamics.
 We call this model flat potential SSERW because of the linear shape of the corresponding potential function $V(n)$ (see Figure \ref{fig:potential}).
In this example the state $\theta$ corresponds to the  transition  probability of jumping one step to the right. 
At time $0$ a state is sampled according to $\mu$ and the random walk starts in $0$.  When the walker hits the target set $\{-N,N\}$ the SSMC starts afresh. Note that the hitting time depends on the current state. Specifically,  the quantity $m_N(\theta)$ defined by \eqref{eq:def_m}, denotes here  the expected time taken by a simple random walk started at the origin to hit $\{-N,N\}$, when it has probability $\theta$ to jump to the right (classical Gambler's ruin on $[-N,N]\cap \bZ$ starting from $0$). Thus, it is well known that, for all $\theta\in [0,1]$ 
\begin{align}\label{eq:timeGR}
m_N(\theta)=\begin{cases}
N^2\;,   & \text{ if }   \theta=\frac{1}{2}\;,
\\
\frac{N}{1-2\theta} -\frac{2N}{1-2\theta}\frac{1}{\left(\frac{1-\theta}{\theta}\right)^{N} +1}\;, & \text{ if } \theta \neq \frac{1}{2}\;. 
\end{cases}
\end{align}
As the above equation shows, the expected switching time for $\theta=1/2$ grows faster in $N$ than the expected switching time  for any $\theta \neq 1/2$. This suggests that there might be an emergence of dominant state, namely  the one corresponding to $\theta=1/2$. Specifically, we have that:

\begin{itemize}
    \item If  $\mu$ is discrete with finite support and  $1/2 \in {\rm supp}(\mu)$,  Theorem~\ref{theo:discrete} implies that there is emerging dominance at $\theta=1/2$.
\end{itemize}

For general $\mu$ the question is more delicate. Note that, {although}  the function $m_N$ above is continuous on $[0,1]$, symmetric around $1/2$ and $
\arg\max_{\theta\in [0,1]} m_N(\theta)=1/2$, we cannot apply Theorem~\ref{theo:unique} (since the corresponding limiting function is $h\equiv 0$).

Let us consider the functions $\overline{ m}_N(\theta)=N^{-1}m_N(\theta),\theta\in[0,1],N\ge1$ and observe from
\eqref{eq:timeGR} that $\overline{ m}_N(\theta)$ converges (pointwise) monotonically  to
\begin{equation}
\label{def:limit_density_GR}
\overline{ m}(\theta)=\frac{1}{|1-2\theta|}\;, \quad  \theta\in [0,1]\;,
\end{equation}
(with the convention that $\frac{1}{0}=\infty$). We then obtain:
\begin{itemize}
    \item If $\mu$    is such that $\mathbf{E}{_\mu}(\overline{m})<\infty$,  Theorem~\ref{th:nodominance} implies that the corresponding sequence of probability measures
$\{\widehat{P}_N\}_{N\geq 1}$ converges weakly, as $N\to \infty$, to a probability measure $\mathbf{P_{\mu}}$ on $[0,1]$ , where 
    \[
      \mathbf{P_{\mu}}(d\theta)=
     \frac{\overline{m}(\theta)}{\mathbf{E}_{\mu}[\overline{m}]}\mu(d\theta)\;. 
     \]
Thus, in this case, there is no emergence of dominant states. 
\item If $\mu$    is such that $\mathbf{E}{_\mu}(\overline{m})=+\infty$,  it can be proved\footnote{The proof goes along the same lines as the proof  of Theorem~\ref{theo:unique}.} that  there is emerging dominance at $\theta=1/2$ also.
\end{itemize}
{\begin{remark}\label{rem:opt}
Suppose that ${\rm supp}(\mu)$ is any finite set of real numbers in $(0,1/2)$.  In this case we see that, for any $\theta_1,\theta_2\in{\rm supp}(\mu)$
\[
\frac{m_N(\theta_1)}{m_N(\theta_2)}\rightarrow \frac{1-2\theta_2}{1-2\theta_1}\ne0,
\]
showing that Assumption (1) of Theorem \ref{theo:discrete} does not hold, while Assumption (2) holds. But since  $\mathbf{E}{_\mu}(\overline{m})<\infty$, there is no emergence of dominant states, showing that Assumption (1) cannot be relaxed in general. 
\end{remark}}
\medskip 
The next result states that, for the flat potential SSERW, any  state $\theta\neq 1/2$ cuts-off  asymptotically {\color{blue} (Proof in \ref{sec:proofs_unpredict}).}

 \begin{restatable}{proposition}{propmetaGR}
\label{thm:metastability_GR}
Let  $\tau_N$ denotes the first time the flat-potential random walk hits the set $\{-N, N\}$.
For any $\theta\ne 1/2$, it holds  that
$$
\frac{\tau_{N}}{\bE_{0,\delta_\theta}[\tau_{N}]}\underset{N\to\infty}{\longrightarrow} 1\;,  \quad \text{ $\bP_{0,\delta_\theta}$-almost surely}\;.
$$
\end{restatable}
  
For  $\theta=1/2$, it is known that $\frac{\tau_{N}}{\bE_{0,\delta_{1/2}}[\tau_{N}]}$ converges in distribution to an unbounded random variable whose distribution is known  \cite{erd1946certain} {but is not exponential, meaning that this state is neither metastable  nor asymptotically cut-off.}

\subsection{Single-well potential SSERW}\label{sec:sw}

As in the previous section, let   $\cX_N=[-N,N]\cap \bZ$,  $T_N=\{-N,N\}$ and $x_0=0$.  For $N\ge1$, let $\bfQ_N={\bf Q}$, where  for any $\theta\in\Xi$, $Q^{(\theta)}$ is defined as follows: 
\begin{align*}
\label{def:single_well}
  &\text{ if } x \in [1, N-1] \cap \bZ\; :
\\
& \qquad \qquad  Q^{(\theta)}(x,y)=
\begin{cases}
\theta\;, & \text{if} \ y=x+1\;, \\
1-\theta\;, & \text{if} \ y=x-1\;, \\
0\;, & \text{otherwise}\;,
\end{cases}
\\
&\text{ if } x \in [ -N+1, 0] \cap \bZ\;  : 
\\
& \qquad \qquad Q^{(\theta)}(x,y)=
\begin{cases}
1-\theta\;, & \text{if} \ y=x+1\;, \\
\theta\;, & \text{if} \ y=x-1\;, \\
0\;, & \text{otherwise}\;
,
\end{cases}
\end{align*}
whereas $Q_N^{(\theta)}(x,\cdot)$ can be chosen  arbitrary for $x \in \{-N, N\}$ because it does not influence the dynamics.
 We call this model single-well potential SSERW because of the single peak/through shape of the corresponding potential function $V(n)$ (see Figure \ref{fig:potential}).
Figure~\ref{Pic:single_well}  represents the dynamic of the single-well  walker.  As opposed to the flat potential random walk discussed in the previous section, here a state $\theta$ corresponds  to the probability of increasing by one unit the distance from the origin $x_0=0$. Thus, for small values of $\theta$, the single-well random walk  may take a very long time to reach the target set $\{-N, N\}$, while  if $\theta$ is close to $1$, the random walk  is pushed away from the origin and ``rapidly'' reaches  $\{-N, N\}$.  See Figure \ref{fig:potential}, graphics C and D, for a pictorial representation. This intuition is formalized in the following proposition {\color{blue} (Proof in \ref{sec:proofs_expectations}).}

\begin{figure}[h]
\begin{tikzpicture}[scale=1.0, shorten >=0.5pt,  >=stealth,  semithick]

\tikzstyle{every state}=[scale=0.2,draw, fill]

\node at (-6, 0)   (a) {};
\node at (6, 0)   (A) {} ;
\draw[very thin, dotted] (a) -- (A);


\node[state] at (0, 0)   (0)  {};
\node at (0,-0.3)   {$0$};

\node[state, fill=white] at (1, 0)   (1) {};
\node at (1,-0.3)   {$1$};

\node[state, fill=white] at (2, 0)   (2) {};
\node at (3, 0)   (3) {};
\node[state, fill=white] at (4, 0)   (4) {};
\node[state, fill=white] at (5, 0)   (5) {};
\node at (5,-0.3)   {$N$};
\node[state] at (-1, 0)   (-1) {};
\node at (-1,-0.3)   {$-1$};
\node[state] at (-2, 0)   (-2) {};
\node at (-3, 0)   (-3) {};
\node[state] at (-4, 0)   (-4) {};
\node[state] at (-5, 0)   (-5) {};
\node at (-5.15,-0.3)   {$-N$};

\path[->, dashed] (4)  edge   [bend left]  node [above] {$\theta$} (5);
%
\path[->, dashed] (4)  edge   [bend left]  node [below] {$1-\theta$} (3);

\path[->] (-4)  edge   [bend left]  node [below] {$\theta$} (-5);
%
\path[->] (-4)  edge   [bend left]  node [above] {$1-\theta$} (-3);
\end{tikzpicture}
\caption{Single-well random walk between $-N$ and $N$ in a $\theta$-environment. White states (from $1$ to $N$) have transition probabilities $\theta$ to the right and $1-\theta$ to the left; black states (from $-N$ to $0$) have opposite transition probabilities. }
\label{Pic:single_well}
\end{figure}
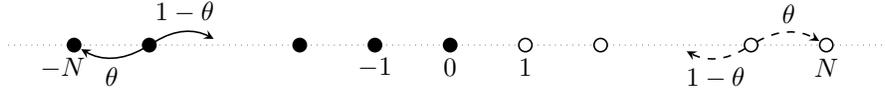

%
%
\begin{restatable}{proposition}{propSW}
\label{Prop:exp_of_tau_SW}
For every integer $N\geq 1$ and $\theta\in (0,1]$, it holds that
\begin{equation}
\label{Expression_for_exp_of_tau_SW}
m_N(\theta)=
\begin{cases}
 N^2\;, & \text{if} \ \theta=1/2\;,\\
\frac{2\theta(1-\theta)}{(1-2\theta)^2}e^{N\log((1-\theta)/\theta)}-\frac{1}{1-2\theta}\left(N+\frac{2\theta(1-\theta)}{1-2\theta}\right)\;, & \text{otherwise}\;.
\end{cases}
\end{equation}
\end{restatable}

\begin{remark}
It turns out that   $m_N(\theta)$ does not depend on the transition probabilities at $0$.  However, these transition probabilities determine the probability 
that the walker hits the state $N$ before $-N$, i.e.,  $\bP_{0,\mu}(X_{\tau_{1}}=N|\Theta_1=\theta)=1-\theta
$. 
\end{remark}

From Proposition \ref{Prop:exp_of_tau_SW} we obtain that  for any $N\ge1$, the function $m_N$ is continuous. Moreover, the function $h_N(\theta)=N^{-1}\log(m_N(\theta))$ is continuous  and monotonically decreasing, converging pointwise on any compact subset of $(0,1]$ to
\begin{equation}
\label{def:limit_potential_SW}
h(\theta)= \begin{cases}
\log \left((1-\theta)/\theta\right)\;, & \theta < 1/2\;,
\\
0\;,  & \theta\geq 1/2\;.  
\end{cases}    
\end{equation}
The function $h$ is continuous and non-increasing on $[0,1]$. Finally,  by Lemma \ref{Lemma:SW_converges_uniformly} in \ref{app:lemmas}, the convergence is uniform on $[a, 1]$ for any fixed $a\in (0,1)$.  
Hence, we obtain that: 

\begin{itemize}
    \item If $a^*:= \sup\{a>0: \mu([0,a))=0\}\in(0,1/2)$, Theorem~\ref{theo:unique} implies that there is emerging dominance at $\theta=a^*$.
    \end{itemize}

On the contrary, if $a^*:= \sup\{a>0: \mu([0,a))=0\}\ge1/2$, the assumption of  Theorem~\ref{theo:unique} does not hold.  We therefore analyze $\overline{m}_N(\theta) = N^{-1}m_N(\theta), N\geq 1$, observing that it monotonically converges (pointwise) to
\[
\overline{m}(\theta):=\frac{1}{2\theta-1}\;, \quad  \theta\ge1/2\;,
\]
(with the convention that $\frac{1}{0}=\infty$).
We then obtain that:
\begin{itemize}
\item If $a^*:= \sup\{a>0: \mu([0,a))=0\}>1/2$, we have that $\mathbf{E}_{\mu}[\overline{m}]<+\infty$, and Theorem~\ref{th:nodominance} implies that the corresponding sequence of probability measures
$\{\widehat{P}_N\}_{N\geq 1}$ converges weakly, as $N\to \infty$, to a probability measure $\mathbf{P_{\mu}}$ on $[0,1]$ , where 
    \[
      \mathbf{P_{\mu}}(d\theta)=
     \frac{\overline{m}(\theta)}{\mathbf{E}_{\mu}[\overline{m}]}\mu(d\theta)\;. 
     \]
In this case, there is no emergence of dominant states.   
\item If $a^*:= \sup\{a>0: \mu([0,a))=0\}=1/2$, we further split in two cases:
\begin{itemize}\item If $\mathbf{E}_{\mu}[\overline{m}]<+\infty$, Theorem~\ref{th:nodominance} implies that 
$\{\widehat{P}_N\}_{N\geq 1}$ converges weakly, as $N\to \infty$, to $      \mathbf{P_{\mu}}(d\theta)=
     \frac{\overline{m}(\theta)}{\mathbf{E}_{\mu}[\overline{m}]}\mu(d\theta)$, and   there is no emergence of dominant states.
\item If  $\mathbf{E}_{\mu}[\overline{m}]=+\infty$, it can be shown\footnote{The proof goes along the same lines as the proof  of Theorem~\ref{theo:unique}.} that there is emerging dominance at $\theta=1/2$. 
\end{itemize}
\end{itemize}

\medskip 
As far as the metastability in the single-well SSERW is concerned, we have the following result {\color{blue} (Proof in \ref{sec:proofs_unpredict}).} 

\begin{restatable}{theorem}{thmmetaSW}\label{thm:metastabilitySW}
Let $\tau_N$ denote the first time the single-well random walk hits the target set $\{-N, N\}$. Then
\begin{enumerate}[i)]
\item For any $0<\theta<1/2$, it holds  that
$$
\frac{\tau_N}{\bE_{0,\delta_\theta} [\tau_N]}\underset{N\to\infty}{\longrightarrow} \mExp(1) \ \text{(in distribution $\bP_{0,\delta_\theta}$)}\;.
$$
\item For any $1/2<\theta\leq 1$, it holds that
$$
\frac{\tau_N}{\bE_{0,\delta_{\theta}}[\tau_N]}\underset{N\to\infty}{\longrightarrow} 1\;, \quad   \text{ $\bP_{0,\delta_\theta}$-almost surely}\;.
$$
\end{enumerate}
\end{restatable}
{Thus, states $\theta \in (0,1/2)$ are metastable, while states $\theta \in (1/2,1]$ are asymptotically {cut-off}.  Interestingly, depending on the support of $\mu$, the chain may alternate between these two phenomena  when $N$ is sufficiently large. 
Observe that the case $\theta=1/2$ is exactly the same  as $\theta=1/2$ in the flat potential SSERW seen before; hence, also here the state $\theta=1/2$ is none of the above. }

\subsection{Alternating-wells potential SSERW}\label{sec:aw}

In this section, building upon  the  single-well random walk, we introduce a simple model whose main feature consists in the emergence of {\em two} dominant states (on the contrary of the previous examples where we always  had either emergence of a unique dominant state or no emergence at all). 

Let   $\cX_N=[-2N,2N]\cap \bZ$,  $T_N=\{-2N,2N\}$ and $x_0=0$.  For $N\ge1$,  let $\bfQ_N$ be such that, for any $\theta\in\Xi$, $Q_N^{(\theta)}$ is defined as follows: \begin{align*}
\label{def:alternating_well}
  &\text{ if } x \in \left([-2N+1, -N-1]  \cup     [N+1, 2N-1]\right) \cap \bZ\; :
\\
& \qquad \qquad \qquad \qquad Q_N^{(\theta)}(x,y)=
\begin{cases}
\theta\;, & \text{if} \ y=x+1\;, \\
1-\theta\;, & \text{if} \ y=x-1\;, \\
0\;, & \text{otherwise}\;,
\end{cases}
\\
& \text{ if } x \in [-N, N] \cap \bZ\;  : 
\\
& \qquad \qquad \qquad \qquad Q_N^{(\theta)}(x,y)=
\begin{cases}
1-\theta\;, & \text{if} \ y=x+1\;, \\
\theta\;, & \text{if} \ y=x-1\;, \\
0\;, & \text{otherwise}\;,
\end{cases}
\end{align*}
whereas $Q_N^{(\theta)}(x,\cdot)$ can be chosen  arbitrary for $x \in \{-2N, 2N\}$ because it does not influence the dynamics.
 We call this model alternating-well potential SSERW because of the double peak/through shape of the corresponding potential function $V(n)$ (see Figure \ref{fig:potential}).
The dynamic is depicted  in Figure~\ref{Pic:alternating_wells}. Intuitively, for large values of $\theta$, the random walk has a bias towards reaching  state $-N$. Once in $-N$, since $\theta$ is large the walker spends a lot of time around $-N$ before reaching the target state $-2N$ or coming back to state $0$. If $\theta$ is small, the situation is analogous  around the state $N$, rather than $-N$. In essence, for large (\emph{resp.} small) values of $\theta$,  the state $-N$ (\emph{resp.} $N$) plays the same role as the origin in the single-well random walk in Section
~\ref{sec:sw}. However, the random walk can now alternate between $-N$ and $N$ depending on the value of $\theta$. See Figure \ref{fig:potential}, graphics E and F, for a pictorial representation.

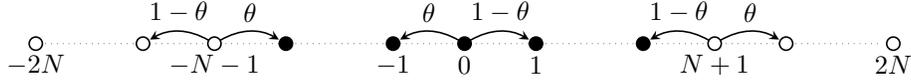
\begin{figure}[h]
\begin{tikzpicture}[scale=0.95, shorten >=0.5pt,  >=stealth,  semithick
]

\tikzstyle{every state}=[scale=0.2,draw, fill]

\node at (-6, 0)   (a) {};
\node at (6, 0)   (A) {} ;
\draw[very thin, dotted] (a) -- (A);


\node[state] at (0, 0)   (0)  {$0$};
\node at (0,-0.3)   {$0$};

\node[state] at (1, 0)   (1) {};
\node at (1,-0.3)   {$1$};

\node at (2, 0)   (2) {};
\node[state] at (2.5, 0)   (3) {};
\node[state, fill=white] at (3.5, 0)   (4) {};
\node at (3.5,-0.3)   {$N+1$};
\node[state, fill=white] at (4.5, 0)   (5a) {};
\node at (4.5, 0)   (5b) {};
\node[state, fill=white] at (6, 0)   (6) {};
\node at (6,-0.3)   {$2N$};

\node[state] at (-1, 0)   (-1) {};
\node at (-1,-0.3)   {$-1$};

\node at (-2, 0)   (-2) {};
\node[state] at (-2.5, 0)   (-3) {};
\node[state, fill=white] at (-3.5, 0)   (-4) {};
\node at (-3.5,-0.3)   {$-N-1$};
\node[state, fill=white] at (-4.5, 0)   (-5a) {};
\node at (-4.5, 0)   (-5b) {};
\node[state, fill=white] at (-6, 0)   (-6) {};
\node at (-6,-0.3)   {$-2N$};

\path[->] (0)  edge   [bend left]  node [above] {$1-\theta$} (1);
\path[->] (0)  edge   [bend right]  node [above] {$\theta$} (-1);

\path[->] (-4)  edge   [bend left]  node [above] {$\theta$} (-3);
\path[->] (-4)  edge   [bend right]  node [above] {$1-\theta$} (-5a);

\path[->] (4)  edge   [bend right]  node [above] {$1-\theta$} (3);
\path[->] (4)  edge   [bend left]  node [above] {$\theta$} (5a);

\end{tikzpicture}
\caption{Alternating-wells random walk between $-2N$ and $2N$ in a   $\theta$-environment. White states (from $-2N$ to $-N-1$ and from $N+1$ to $2N$) have transition probabilities $\theta$ to the right and $1-\theta$ to the left; black states (from $-N$ to $N$) have opposite transition probabilities. }
\label{Pic:alternating_wells}
\end{figure}

%

In the following proposition we compute the expected switching time corresponding to each state $\theta$, for any $N\ge1$ {\color{blue} (Proof in \ref{sec:proofs_expectations}).}

\begin{restatable}{proposition}{propAW}
\label{Prop:esp__of_tau_alternating_wells}
For any integer $N\geq 1$ and $\theta \in (0,1)$, it holds that
\begin{equation*}
m_N(\theta)=
\begin{cases}
4N^2\;, & \text{if} \ \theta=1/2\;,
\\
\frac{\frac{2\theta(1-\theta)}{(1-2\theta)^2}}{ \theta \left(\frac{\theta}{1-\theta} \right)^{N} + 1-\theta }\left( 1-\left(\frac{\theta}{1-\theta}\right)^N\right) \left( \left(\frac{1-\theta}{\theta}\right)^N-   \left(\frac{\theta}{1-\theta} \right)^{N}
\right)\;, 
&  \text{otherwise}\;.
\end{cases}
\end{equation*}
\end{restatable}

From Proposition \ref{Prop:esp__of_tau_alternating_wells} we obtain that the function $m_N$ is continuous and symmetric around $1/2$, for every $N\ge1$. Moreover, the function $h_N(\theta)=N^{-1}\log(m_N(\theta))$ is continuous, monotonically decreasing on $(0,1/2]$ and  monotonically increasing on $[1/2,1)$. It converges pointwise on  $(0,1)$ to
\begin{equation}
\label{def:limit_potential_AW}
h(\theta)=|
\log \left((1-\theta)/\theta\right)|  \;,   
\end{equation}
and,  by Lemma \ref{Lemma:AW_converges_uniformly} in \ref{app:lemmas}, the convergence is uniform on $[a, b]$ for any fixed $a\in (0,1/2)$ and $b \in (1/2,1)$.
The function $h$ is monotonically decreasing on $(0,1/2]$ and monotonically increasing on $[1/2,1)$.  

Let us  define
$a^*=\sup\{a\geq 0: \mu([0,a])=0\}$ and $b^*=\inf\{b\geq a^*: \mu([a^*,b])=1\}$; thus  $\text{supp}(\mu)\subseteq [a^*,b^*]$. We have the following situations. 
\begin{itemize}
    \item If $a^*\neq 1-b^*$, Theorem~\ref{theo:unique} implies that there is emerging dominance at $\theta^*=\{a^*\}$ if $a^*<1-b^*$, and at $\theta^*=\{b^*\}$ if $a^*>1-b^*$. 
  \item {If $a^*=1-b^*$,  then $\arg\max_{\theta\in[a^*,b^*]}h(\theta)$ is not unique, and Theorem~\ref{theo:unique} does not apply. In this case, if we assume that $\mu$  has continuous density $g_{\mu}$ with respect to the Lebesgue measure, we can use Theorem \ref{theo:twodominance}. First notice that, due to the symmetry of the functions $m_N$ and $h$ in \eqref{def:limit_potential_AW}, it holds that $h(a^*)=h(b^*)$, $-h'(a^*)=h'(b^*)>0$ and $d_1=d_2$. In this case,   there is emerging dominance at $\{a^*,b^*\}$ with weights
\[
w_{a^*}=\frac{g_{\mu}(a^*)}{g_{\mu}(a^*)+g_{\mu}(b^*)}\;,  \quad w_{b^*}=1 - w_{a^*}\;.
\]
}
  
  In passing, note that if the density $g_\mu$ is constant (i.e., $\mu$ is uniform on $[a^*, b^*]$), then $w_{a^*}=w_{b^*}=1/2$.


 \end{itemize}

We conclude this section with the following result stating that, in the alternating-wells SSERW,  any state  $\theta\neq 1/2$ is  metastable {\color{blue} (Proof in \ref{sec:proofs_unpredict}).}

\begin{restatable}{theorem}{thmmetaAW}\label{th:meta_aw} Let $\tau_N$ denote the first time the alternating-wells random walk hits the target set $\{-2N, 2N\}$. Then, 
for any $\theta\in [0,1]\setminus\{1/2\}$, it holds that 
$$
\frac{\tau_N}{\bE_{0,\delta_\theta}[\tau_N]}\underset{N\to\infty}{\longrightarrow} \mExp(1) \ \text{(in distribution $\bP_{0,\delta_\theta}$)}\;.
$$
\end{restatable}

For $\theta=1/2$, the alternating-wells random walk is equivalent to the single-well (as well as flat potential).

\medskip 
Figure~\ref{fig:potential} depicts intuitively the potentials corresponding to different states in the three  examples: flat-potential, single-well and alternating-wells potential. 

\begin{figure}[h]
\centering
\includegraphics[scale = 0.7]{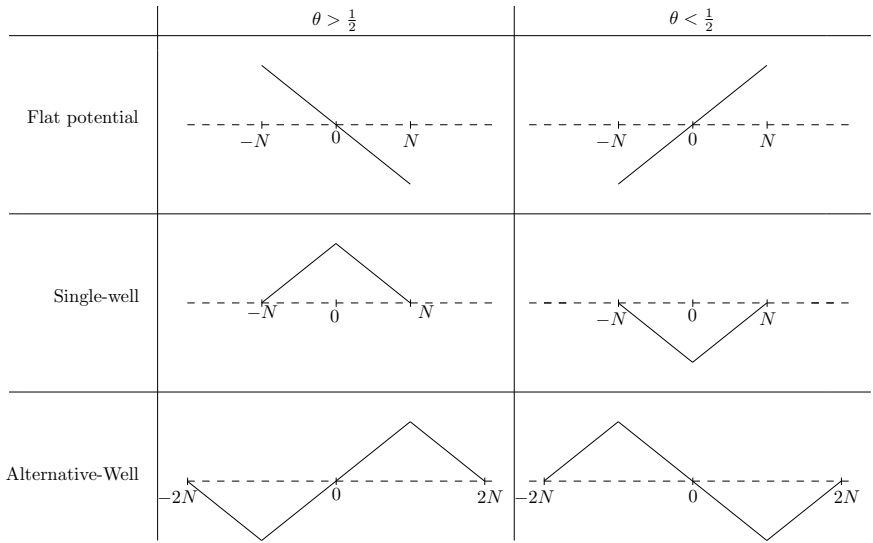}
\caption{ Rough representation of the potential for different models, according to definition \eqref{eq:potential}.}
\label{fig:potential}
\end{figure}

\section{Final Comments}\label{sec:final}
 
In this article, we introduced the class of SSMCs. A SSMC is a bivariate Markov chain where the interaction between the components is both ways.  
On one hand, the first component (the location process) is driven by the second component (the state process). On the other hand,  the second component changes its current value at random times  corresponding to hitting times of the first component. 
For this class of models, we exhibited some conditions under which we observe the emerging dominance phenomenon. These conditions also allow us to characterize the dominant states. Our results were illustrated through the SSERWs (see, Section~\ref{sec:examples}), simple examples of SSMCs.
A few natural extensions to the SSMCs and results introduced in this paper could be considered. 

\begin{enumerate}[(a)]
\item\label{b} {\bf Changing the starting point.}    In our model, every time the location process $\{X_i\}_{i\in \mathbb{N} }$  hits the  target set $T$, the SSMC 
$\{(X_i, \eta_i)\}_{i\in \mathbb{N}}$ start afresh according to the distribution 
$\delta_{x_0} \times \mu$; that is the location process is set to the origin $x_0$ and a new state is sampled according to $\mu$ (independently). 
One could consider a variant in which when the SSMC starts afresh, the new starting point of the chain is chosen according to a distribution $\nu \times \mu$, or even according to $\nu^x \times \mu$, with $x \in T$. These cases would account for the situation in which the origin is sampled according to a distribution $\nu$, or even according to a distribution $\nu^x$  which depends on the target point  $x \in T$, which was hit in the previous ``round'' of the SSMC.
Note however, that when the starting point of the chain is chosen according to $\nu^x \times \mu$, with $x \in T$, the inter-switching times would not necessarily be i.i.d.

\item{\bf Markovian dependence of the parameter process.} In SSMC, the state process is kept constant during inter-switching times. A possible variation may consider a Markovian dependence on the state process.  This could be done by using a transition kernel $\mathcal K:\Xi\times\Xi\rightarrow[0,1]$, and now, $\mu$ would represent the initial distribution of this  Markov process.
This would introduce a stronger dependence between the location process and the state process and   it is not clear whether the emergence of dominant states still holds. 
 
\item {\bf Convergence of the state process.} We showed the metastability and cut-off only for a fixed value of the state process. This is not very satisfactory as we do not know if similar convergence holds for the entire process. The main difficulty is that different states have distinct normalizing constant, therefore, a single normalizing constant for the state process does not result in a interesting limit process. On the other hand, considering a case of multiple metastable dominant states with comparable $m_N(\theta_i^*)$ (as in the alternated wells SSERW for example),  the normalizing constants $\mathbb E_{x_0,\delta_{\theta^*_i}}(\tau_N)$ would be comparable too. So, rescaling  time by their order may yield a jump process between dominant states with exponential waiting times with parameter depending on the current state. This is a question we are planning to investigate in the future. 
\end{enumerate}

\subsection{A biological consequence}
Although animals typically have a large degree of freedom for producing movements, only relatively few behaviors are observed in Nature. For example, humans can create many different postures, but generally, we only observe few types of behaviors. The relatively small number of observed behaviors is what allows ethologists to study animal behavior \cite{brown2018ethology}. The standard explanation for the predominance of a specific behavior is usually associated with the animal's predisposition to produce the behavior or some learning mechanism that makes the behavior more frequent. In our model, we can think of modeling the predisposition or learning by a fixed or history-dependent state distribution $\mu$, respectively. Our results suggest that a third explanation exists: Emergence of dominance state by the typical duration of the behavior - \emph{the longer it lasts the merrier}.

\section*{Acknowledgements}
 The authors  thank the anonymous referee for the helpful remarks and corrections. This research has been conducted as part of FAPESP project {\em Research, Innovation and
Dissemination Center for Neuromathematics} (grant 2013/07699-0).
SG thanks CNPq (Universal 439422/2018-3) and FAPESP (2019/23439-4) for financial support, and DME-UFRJ for hospitality during part of the redaction of this paper. GI and GO thank PIPGE-UFSCar (CAPES) for supporting a short visit to UFSCar where substantial part of this paper were done. 
DYT thanks PIPGE-UFSCar (CAPES) and IM-UFRJ (CAPES) for supporting the visits that gave the opportunity to collaborate on this work.

\newpage
\appendix

\section{Proofs}\label{app:proofs}

The proofs  are separated by ``themes'' so that proofs using similar arguments appear together. There are three main themes: 
\begin{itemize}
\item asymptotic distribution of $\widehat{P}_N,N\ge1$,
 in \ref{sec:proofs_as_distribution} (proofs of Proposition~\ref{Prop:P_hat_N_n_convto_P_hat_N}, Theorems \ref{theo:discrete}, \ref{theo:unique} and  \ref{theo:twodominance} and Remark  \ref{cor:1d_K2});
\item expectation of hitting times  $\tau_N,N\ge1$, in \ref{sec:proofs_expectations} (proofs of Propositions \ref{Prop:exp_of_tau_SW} and \ref{Prop:esp__of_tau_alternating_wells});
\item metastability/cut-off, in \ref{sec:proofs_unpredict} (proofs of Proposition \ref{thm:metastability_GR} and Theorems \ref{thm:metastabilitySW} and \ref{th:meta_aw}).
\end{itemize}
Finally, in  \ref{app:lemmas}, we give some auxiliary lemmas used in the proofs.

\subsection{$\widehat{P}_N$ and its asymptotic analysis}
\label{sec:proofs_as_distribution}

\medskip 

\propLLN*
\begin{proof}[Proof of Proposition~\ref{Prop:P_hat_N_n_convto_P_hat_N}]
To avoid clutter, we shall denote $\tau_i(\Theta_i)$ simply as $\tau_i$. If we define  $M_n=\sup\{k\geq 1: \tau_{1}+\ldots+\tau_{k}\leq n\}$, for $n\geq 1$, the empirical measure $\hat{p}_{n}(A)$ can be written as
$$
\hat{p}_{n}(A)=\frac{1}{n}\sum_{i=1}^{M_n}\tau_{i} \1_{(\Theta_i\in A)}
 +\frac{1}{n}\left(n-\sum_{i=1}^{M_n}\tau_i\right)\1_{(\Theta_{M_{n}+1}\in A)}
\;.
$$
Moreover,
$\sum_{i=1}^{M_n}\tau_{i}\leq n< \sum_{i=1}^{M_n+1}\tau_{i}$, which implies
\begin{equation}
\label{ineq_1:Prop:P_hat_N_n_convto_P_hat_N}
\frac{\sum_{i=1}^{M_n}\tau_{i} \1_{(\Theta_i\in A)}}{\sum_{i=1}^{M_n+1}\tau_{i}}\leq \widehat{p}_{n}(A)\leq \frac{\sum_{i=1}^{M_n}\tau_{i} \1_{(\Theta_i\in A)}}{\sum_{i=1}^{M_n}\tau_{i}} +\frac{\left(n-\sum_{i=1}^{M_n}\tau_i\right)\1_{(\Theta_{M_{n}+1}\in A)}}{n}\;.
\end{equation}
As mentioned in Remark~\ref{rem:iid} the random variables  $\{\tau_{i}\}_{i\geq 1}$  are i.i.d. and under Assumption~\ref{Ass:1} they have finite expectation $\mathbf{E}_{\mu}[m]$. In particular, this implies  that $M_n\to\infty$, $\bP_{x_0, \mu}$-almost surely as $n\to\infty$. Since the $\Theta_i,i\ge1$ are also i.i.d., it follows that $\1_{(\Theta_i\in A)}\tau_{i},i\geq 1$ are i.i.d. with expectation $\mathbf{E}_{\mu}[\1_Am]$. Thus, by the Strong Law of Large Numbers,
 the term of the lhs and first term of rhs of Equation~\eqref{ineq_1:Prop:P_hat_N_n_convto_P_hat_N} converge to $\widehat{P}(A)$, while the second term of the rhs vanishes, since $\frac{\left(n-\sum_{i=1}^{M_n}\tau_i\right)\1_{(\Theta_{M_{n}+1}\in A)}}{n}\leq \frac{\left(n-\sum_{i=1}^{M_n}\tau_i\right)}{n}\leq \frac{\tau_{M_n +1}}{n}$ and $\tau_{M_n +1}/n\to 0$ almost surely.

\end{proof}

\subsubsection{A key Lemma for proving emerging dominance}

Before we give the proofs of the theorems concerning the asymptotic distribution of $\widehat{P}_N,N\ge1$, we state and prove a general  ``Key Lemma'', which gives the more general conditions under which we were able to prove emerging dominance. Before we can state it, we need some further notation. 

First, recall that, for any $N\ge1$, $\widehat{P}_N(A):=\frac{\mathbf{E}_{\mu}\left(\1_Am_N\right)}{\mathbf{E}_{\mu}(m_N)}$ for any measurable set $A\subset\Xi$. 
We write ${\rm supp}(\mu)$ to indicate the support of a probability measure $\mu$. For any $r>0$ and $\theta\in\Xi$, we denote $B_r(\theta)=\{\hat\theta \in \Xi: d(\hat\theta, \theta)<r\}$ the ball of radius $r$ centerred at $\theta$  {according to some metric $d$ on $\Xi$}. For any measurable set $V\subset\Xi$, we define
\[
I_N [V]:=\int_{V} m_N(\theta)\mu(d\theta)\;.
\]

\begin{lemma}[Key Lemma]\label{th:main}
Let $(\Xi,d)$ be a compact metric space and $\mu$ a probability measure on $\Xi$. 
%
Assume that there are $\theta^*_1,\ldots, \theta^*_K\in \Xi\cap{\rm supp}(\mu)$ satisfying
the following conditions:
\begin{itemize}
    \item[(A)] For every $\delta>0$ sufficiently small there exists $1\leq i\leq K$ such that
    $$
    \frac{I_N\left[\left(\bigcup_{j=1}^K B_{\delta}(\theta^*_j) \right)^c\right]}{I_N\left[B_{\delta}(\theta^*_i)\right]} \underset{N\to \infty}{\longrightarrow} 0\;.
    $$
    \item[(B)] For every $\delta>0$ sufficiently small and for each $1\leq i,j\leq K$,  $i\neq j$, it holds
    $$
    C_{1,i,j}(\delta)\leq \liminf_{N\to\infty}\frac{I_N\left[B_{\delta}(\theta^*_j)\right]}{I_N\left[ B_{\delta}(\theta^*_i)\right]}\leq   \limsup_{N\to\infty}\frac{I_N\left[B_{\delta}(\theta^*_j)\right]}{I_N\left[ B_{\delta}(\theta^*_i)\right]}\leq C_{2,i,j}(\delta)\;.$$
    Moreover, the limits $\lim_{\delta\to 0}C_{1,i,j}(\delta)$ and $\lim_{\delta\to 0}C_{2,i,j}(\delta)$ exist and $$\lim_{\delta\to 0}C_{1,i,j}(\delta)=\lim_{\delta\to 0}C_{2,i,j}(\delta)=:c_{i,j}>0\;.$$
\end{itemize}
%
Then, there is emerging dominance at $\{\theta^*_1,\ldots,\theta^*_K\}$ with weights
$$
 w_{i}= \frac{1}{1 + \sum_{\substack{j=1\\ j \neq i}}^K c_{i,j}}\;,\quad i=1,\ldots,K.
$$
\end{lemma}
%
 
\begin{remark}
\label{rmk_conditionA_prime}
One can easily check that a sufficient condition which does not depend on the probability measure $\mu$ implying  condition $(A)$ in Key Lemma~\ref{th:main} is the following:
\begin{itemize}
    \item[(A')] There exist $\theta^*_1,\ldots, \theta^*_K\in \Xi\cap{\rm supp}(\mu)$ such that for every $\delta_1>0$ sufficiently small there exists $0<\delta_2<\delta_1$ and $1\leq i\leq K$ such that
    $$
    \frac{\sup_{\theta\in \left(\bigcup_{j=1}^K B_{\delta_1}(\theta^*_j) \right)^c  }m_N(\theta)}{\inf_{\theta\in B_{\delta_2}(\theta^*_i)}m_N(\theta)}\underset{N\to \infty}{\longrightarrow} 0.
    $$
\end{itemize}
\end{remark}

For any continuous function $g:\Xi\to\bR$, we define 
$$
\widehat{P}_N(g):=\frac{\mathbf{E}_{\mu}(g m_N)}{\mathbf{E}_{\mu}(m_N)}\;.
$$
We shall also write, for any continuous function $g:\Xi\to\bR$ and any measurable set $V\subset\Xi$
\[
I_N (g, V):=\int_{V} g(\theta) m_N(\theta)\mu(d\theta)\;.
\]
When $g\equiv 1$, we have $I_N(1,V)=I_N[V]$. 
\begin{proof}[{\bf Proof of {Lemma} \ref{th:main}}]
We need to show that for any continuous function $g:\Xi\to\bR$, we have
$$
\widehat{P}_N(g)\underset{N\to\infty}{\longrightarrow} \sum_{i=1}^Kg(\theta^*_i)w_i.
$$
First, notice that for any $\delta>0$ sufficiently small, we can decompose $\widehat{P}_N(g)$ as
$$
\widehat{P}_N(g)=\frac{\sum_{i=1}^KI_N (g,B_\delta(\theta^*_i))+I_N (g,F_{\delta,K})}{\sum_{j=1}^KI_N [B_\delta(\theta^*_j)]+I_N[F_{\delta,K}]}\;.
$$
where $F_{\delta,K}=\left(\bigcup_{i=1}^K B_{\delta}(\theta^*_i)\right)^c$. Now by continuity of $g$, given $\varepsilon>0$, there exists $\delta>0$ small enough such that for each $1\leq i\leq K$ and all $\theta \in B_\delta(\theta^*_i)$,
$$
g(\theta^*_i)-\varepsilon <g(\theta)<g(\theta^*_i)+\varepsilon\;,
$$
which implies that $(g(\theta^*_i)-\varepsilon)I_N [B_\delta(\theta^*_i)]\leq I_N (g, B_\delta(\theta^*_i))\leq (g(\theta^*)+\varepsilon)I_N [B_\delta(\theta^*_i)]$ for all  $1\leq i\leq K.$ Hence, it follows that
$$
\widehat{P}_N(g)\geq \frac{\sum_{i=1}^K(g(\theta_i^*)-\varepsilon)I_N [B_\delta(\theta^*_i)]+\inf_{\theta\in F_{\delta,K}}g(\theta)\;I^N[F_{\delta,K}]}{\sum_{j=1}^KI_N [B_\delta(\theta^*_j)]+I_N[F_{\delta,K}]}\;,
$$
and
$$
\widehat{P}_N(g)\leq \frac{\sum_{i=1}^K(g(\theta_i^*)+\varepsilon)I_N [B_\delta(\theta^*_i)]+\sup_{\theta\in F_{\delta,K}}g(\theta)\;I^N[F_{\delta,K}]}{\sum_{j=1}^KI_N [B_\delta(\theta^*_j)]+I_N[F_{\delta,K}]}\;.
$$
Note that, Condition (A) implies that
$$
\frac{I_N[F_{\delta,K}]}{\sum_{j=1}^KI_N [B_\delta(\theta^*_j)]+I_N[F_{\delta,K}]}\leq \frac{I_N[F_{\delta,K}]}{\sum_{j=1}^KI_N [B_\delta(\theta^*_j)]}\underset{N\to\infty}{\longrightarrow} 0\;,
$$
and Condition (B) implies that
\begin{align*}
\frac{1}{1+\sum_{j=1:j\neq i}^KC_{2,i,j}(\delta)}&\leq \liminf_{N\to\infty} \frac{I_N [B_\delta(\theta^*_i)]}{\sum_{j=1}^KI_N [B_\delta(\theta^*_j)]}\\   &\leq \limsup_{N\to\infty}\frac{I_N [B_\delta(\theta^*_i)]}{\sum_{j=1}^KI_N [B_\delta(\theta^*_j)]}\leq \frac{1}{1+\sum_{j=1:j\neq i}^KC_{1,i,j}(\delta)}\;. 
\end{align*}
Thus, it follows that for any $\varepsilon>0$ and all $\delta>0$ small enough,
\begin{align*}
\sum_{i=1}^K\frac{g(\theta_i^*)-\varepsilon}{1+\sum_{j=1:j\neq i}^KC_{2,i,j}(\delta)}
& \leq \liminf_{N\to\infty}\widehat{P}_N(g)\\
&\leq\limsup_{N\to\infty}\widehat{P}_N(g)\leq  \sum_{i=1}^K\frac{g(\theta_i^*)+\varepsilon}{1+\sum_{j=1:j\neq i}^KC_{1,i,j}(\delta)}\;.
\end{align*}
By Condition (B), we can take $\delta\to 0$ to conclude that for any $\varepsilon>0$,
$$
\sum_{i=1}^K(g(\theta_i^*)-\varepsilon)w_i
 \leq \liminf_{N\to\infty}\widehat{P}_N(g)\\
\leq\limsup_{N\to\infty}\widehat{P}_N(g)\leq  \sum_{i=1}^K(g(\theta_i^*)+\varepsilon)w_i\;.
$$
By taking $\varepsilon\to 0$, the result follows.
\end{proof}

\subsubsection{Proofs of the theorems on emerging dominance}

\medskip

\thmFinite*
\begin{proof}[{\bf Proof of Theorem \ref{theo:discrete}}]
We need to show that conditions (1) and (2) imply conditions (A) and (B) of Lemma~\ref{th:main}. To that end, observe that for any $\delta>0$ small enough
$$
\frac{I_N\left[F_{\delta,K}\right]}{I_N\left[B_{\delta}(\theta^*_i)\right]}=\frac{I_N\left[F_{\delta,K}\right]}{m_N(\theta^*_i)\mu(\theta^*_i)}=\sum_{\theta\in F_{\delta,K}\cap\text{Supp}(\mu)}\frac{m_N(\theta)\mu(\theta)}{m_N(\theta^*_i)\mu(\theta^*_i)}\;,
$$
where $F_{\delta,K}=\left(\bigcup_{i=1}^K B_{\delta}(\theta^*_i)\right)^c$. Since by Condition (1) each term of the finite sum above goes to $0$ as $N\to\infty$, it follows that  $\frac{I_N\left[F_{\delta,K}\right]}{I_N\left[B_{\delta}(\theta^*_i)\right]}\to 0$ as $N\to\infty$ and Condition $(A)$ holds.

It remains to show that Condition $(2)$ implies $(B)$. To see this, fix $1\leq i,j\leq K$ with $i\neq j$ and take $\delta>0$ small enough in such a way
$$
\frac{I_N\left[B_{\delta}(\theta^*_j)\right]}{I_N\left[ B_{\delta}(\theta^*_i)\right]}=\frac{\mu(\theta^*_j)m_N(\theta^*_j)}{\mu(\theta^*_i)m_N(\theta^*_i)}\;.
$$
Hence, by Condition $(2)$ for all $\delta>0$ small enough
$$
\lim_{N\to\infty }\frac{I_N\left[B_{\delta}(\theta^*_j)\right]}{I_N\left[ B_{\delta}(\theta^*_i)\right]}=\frac{\mu(\theta^*_j)}{\mu(\theta^*_i)}d_{i,j}\;.
$$
In this case, condition (B) holds with  $C_{1,i,j}(\delta)=C_{2,i,j}(\delta)=\frac{\mu(\theta^*_j)}{\mu(\theta^*_i)}d_{i,j}$ for all $\delta>0$ small enough and the result follows.
\end{proof}

\medskip

\thmUnique*
\begin{proof}[{\bf Proof of Theorem \ref{theo:unique}}]
By {Key Lemma~\ref{th:main}}, it suffices to show that
for any open neighborhood  $V\subset\Xi$ of $\theta^*$ we have
\begin{equation}
\label{eq:ration_to_zero_general_case}
\frac{I_N [V^c]}{I_N [V]}\underset{ N\to\infty}{\longrightarrow} 0\;. 
\end{equation}

 Observe that Assumption \ref{Ass:1} implies that $h_N,N\ge1$ are continuous, and by uniform convergence, we conclude that $h$ itself is continuous.

To show \ref{eq:ration_to_zero_general_case}, let $\tilde{\theta}\in \arg\max_{\theta\in V^c} h(\theta)$ (the maximum exists since $h$ is continuous, $V^c$ is closed and $\Xi$ compact) and observe that 
$h(\tilde{\theta})<h(\theta^*)$, by the definition of $\theta^*$. Thus, let $\delta>0$ be such that
$$
h(\tilde{\theta})+\delta<h(\theta^*)-2\delta\;.
$$
Since $h_N$ converges uniformly to $h$, it follows that for all $N$ sufficiently large 
$$
I_N [V^c]\leq \mu(V^c) e^{N\sup_{\theta\in V^c}h_N(\theta)}\leq \mu(V^c)e^{N({h}(\tilde{\theta})+\delta)}\;.
$$
Similarly, we also have for all $N$ sufficiently large, 
$$
I_N [V]\geq \int_{V}e^{N({h}(\theta)-\delta)}\mu(d\theta)\;.
$$
Finally, since ${h}$ is continuous at $\theta^*$, there exists a neighborhood  $V'\subseteq V$ of $\theta^*$ such that  ${h}(\theta)>{h}(\theta^*)-\delta$ for all $\theta\in V'$, which implies that
$$
I_N [V]\geq e^{N({h}(\theta^*)-2\delta)} \mu(V')\;.
$$
Since $\theta^*\in\text{supp}(\mu)$, we have that $\mu(V')>0$.
As a consequence, for all $N$ sufficiently  large, we have that     
\begin{equation*}
\label{eq:ration_to_zero}
\frac{I_N [V^c]}{I_N [V]}\leq \frac{e^{N({h}(\tilde{\theta})+\delta)}\mu(V^c)}{e^{N({h}(\theta^*)-2\delta)} \mu(V')}\;\underset{N\to\infty}{\longrightarrow} 0 \;,  
\end{equation*}
and the result follows.
\end{proof}

\medskip

\remdouble*
\begin{proof}[ {\bf Proof of Remark \ref{cor:1d_K2}}]

The exact same proof of Theorem \ref{theo:unique} ensures Condition (A) of {Key Lemma~\ref{th:main}} holds with either $i=1$ or $i=2$. 

To conclude the proof, we need to show that Condition (B) of 
{Key Lemma~\ref{th:main}} holds with $c_{1,2}=(e^{d_2}g_{\mu}(\theta^*_2)\sqrt{-h''(\theta^*_1)})/(e^{d_1}g_{\mu}(\theta^*_1)\sqrt{-h''(\theta^*_2)}).$ To that end, take $\delta>0$ small enough and observe  that by the Mean Valued Theorem, 
\begin{equation}
\label{ineq_0_proof_cor_3}
\frac{\int_{\theta^*_2-\delta}^{\theta^*_2+\delta}m_N(\theta)g_{\mu}(\theta)d\theta}{\int_{\theta^*_1-\delta}^{\theta^*_1+\delta}m_N(\theta)g_{\mu}(\theta)d\theta}=\frac{g_{\mu}(\xi_2)\int_{\theta^*_2-\delta}^{\theta^*_2+\delta}e^{Nh_N(\theta)}d\theta}{g_{\mu}(\xi_1)\int_{\theta^*_1-\delta}^{\theta^*_1+\delta}e^{Nh_N(\theta)}d\theta}\;,    
\end{equation}
where,  $\theta^*_1-\delta<\xi_1<\theta^*_1+\delta$ and $\theta^*_2-\delta<\xi_2<\theta^*_2+\delta$. 

Since $h''(\theta^*_i)<0$ for each $1\leq i\leq 2$, one can take $\varepsilon>0$ such that $h''(\theta^*_i)+2\varepsilon<0$, for both $i=1$ and $i=2$. 
By the continuity of $h''$, one can find $\delta$ sufficiently small such that 
\begin{equation}
\label{ineq_1_proof_cor_3}
 h''(\theta^*_i)-\varepsilon\leq h''(\theta)\leq h''(\theta^*_i)+\varepsilon,   
\end{equation}
for all $\theta\in [\theta^*_i-\delta,\theta^*_i+\delta]$ and both $i=1$ and $i=2$.

The uniform convergence of $h''_N$ to $h''$ around $\theta^*_i$ and \eqref{ineq_1_proof_cor_3} ensure that there exists $N_0=N_0(\delta,\varepsilon)$ such that for all $N\geq N_0$, 
\begin{equation}
\label{ineq_2_proof_cor_3}
h''(\theta^*_i)-2\varepsilon \leq h''_N(\theta)\leq h''(\theta^*_i)+2\varepsilon,    
\end{equation}
for all $\theta\in [\theta^*_i-\delta,\theta^*_i-\delta]$ and $1\leq i\leq 2$.

Now, since $h_N\in C^2([a,b])$, for each $1\leq i\leq 2$, we can expand $h_N$ around $\theta^*_i$ up to order 2 to deduce that 
for all $\theta\in [\theta^*_i-\delta,\theta^*_i-\delta]$, 
$$
h_N(\theta)=h_N(\theta^*_i)+h'_N(\theta^*_i)(\theta-\theta^*_i)+h_N''(\zeta_{i,\theta})(\theta-\theta^*_i)^2/2\;,
$$
where $\zeta_{i,\theta}$ is between $\theta$ and $\theta^*_i$. By using that for all $N$ sufficiently large $|h'_N(\theta^*_i)|\leq C/N$ and \eqref{ineq_2_proof_cor_3}, we can suppose (by increasing $C$ if necessary) that for all $N\geq N_0$, $\theta\in [\theta^*_i-\delta,\theta^*_i-\delta]$ and $1\leq i\leq 2,$ we have that
\begin{align}
\label{ineq_3_proof_cor_3}
-(C\delta)/N+(h''(\theta^*_i)-2\varepsilon)(\theta-\theta^*_i)^2/2&\leq h_N(\theta)-h_N(\theta^*_i) \nonumber\\
\leq & (C\delta)/N+(h''(\theta^*_i)+2\varepsilon)(\theta-\theta^*_i)^2/2.
\end{align}
By observing that $N(h_N(\theta^*_i)-h(\theta^*_i))\to d_i$ and by increasing $N_0$ if necessary, we deduce from \eqref{ineq_3_proof_cor_3} that the following inequalities hold for all $N\geq N_0$, $\theta\in [\theta^*_i-\delta,\theta^*_i-\delta]$ and $1\leq i\leq 2:$  
\begin{equation}\label{eq:1.}
N(h_N(\theta)-h(\theta^*_i))\leq d_i+\varepsilon+ C\delta+N(h''(\theta^*_i)+2\varepsilon)(\theta-\theta^*_i)^2/2.
\end{equation}
and
\begin{equation}\label{eq:2.}
N(h_N(\theta)-h(\theta^*_i))\geq d_i-\varepsilon- C\delta+N(h''(\theta^*_i)-2\varepsilon)(\theta-\theta^*_i)^2/2.
\end{equation}
Hence, we deduce from the inequalities above that for all $N\geq N_0$ (recall that $m_N(\theta)=e^{Nh_N(\theta)}$ and that $h(\theta^*_1)=h(\theta^*_2)$ by assumption), it holds
\begin{align}\label{eq:*}
\frac{\int_{\theta^*_2-\delta}^{\theta^*_2+\delta}m_N(\theta)d\theta}{\int_{\theta^*_1-\delta}^{\theta^*_1+\delta}m_N(\theta)d\theta}&\leq \frac{e^{(d_2+\varepsilon+\delta C)}\int_{\theta^*_2-\delta}^{\theta^*_2+\delta}e^{N(h''(\theta^*_2)+2\varepsilon)(\theta-\theta^*_2)^2/2}d\theta}{e^{(d_1-\varepsilon-\delta C)}\int_{\theta^*_1-\delta}^{\theta^*_1+\delta}e^{N(h''(\theta^*_1)-2\varepsilon)(\theta-\theta^*_1)^2/2}d\theta}\;.
\end{align}
By making the change of variables $x=\sqrt{N}(\theta-\theta^*_i)$ in the right-most integrals above, we obtain
\begin{equation}
\label{ineq_4_proof_cor_3}
\frac{\int_{\theta^*_2-\delta}^{\theta^*_2+\delta}e^{N(h''(\theta^*_2)+2\varepsilon)(\theta-\theta^*_2)^2/2}d\theta}{\int_{\theta^*_1-\delta}^{\theta^*_1+\delta}e^{N(h''(\theta^*_1)-2\varepsilon)(\theta-\theta^*_1)^2/2}d\theta}=\frac{\int_{-\delta\sqrt{N}}^{\delta\sqrt{N}}e^{(h''(\theta^*_2)+2\varepsilon)x^2/2}dx}{\int_{-\delta\sqrt{N}}^{\delta\sqrt{N}}e^{N(h''(\theta^*_1)-2\varepsilon)x^2/2}dx}\;.    
\end{equation}
Therefore, from \eqref{ineq_0_proof_cor_3}, \eqref{eq:*} and \eqref{ineq_4_proof_cor_3}, we have that for all $N\geq N_0$, 
$$\frac{\int_{\theta^*_2-\delta}^{\theta^*_2+\delta}m_N(\theta)g_{\mu}(\theta)d\theta}{\int_{\theta^*_1-\delta}^{\theta^*_1+\delta}m_N(\theta)g_{\mu}(\theta)d\theta}\leq \frac{g_{\mu}(\xi_2)e^{(d_2+\varepsilon+\delta C)}\int_{-\delta\sqrt{N}}^{\delta\sqrt{N}}e^{(h''(\theta^*_2)+2\varepsilon)x^2/2}dx}{g_{\mu}(\xi_1)e^{(d_1-\varepsilon-\delta C)}\int_{-\delta\sqrt{N}}^{\delta\sqrt{N}}e^{N(h''(\theta^*_1)-2\varepsilon)x^2/2}dx}\;,
$$
where, $\theta^*_1-\delta<\xi_1<\theta^*_1+\delta$ and $\theta^*_2-\delta<\xi_2<\theta^*_2+\delta$.
In particular, it follows that 
$$
\limsup_{N\to\infty}\frac{\int_{\theta^*_2-\delta}^{\theta^*_2+\delta}m_N(\theta)g_{\mu}(\theta)d\theta}{\int_{\theta^*_1-\delta}^{\theta^*_1+\delta}m_N(\theta)g_{\mu}(\theta)d\theta}\leq \frac{g_{\mu}(\xi_2)e^{(d_2+\varepsilon+\delta C)}}{g_{\mu}(\xi_1)e^{(d_1-\varepsilon-\delta C)}}\sqrt{\frac{2\varepsilon-h''(\theta^*_1)}{-2\varepsilon-h''(\theta^*_2)}}.
$$
By similar arguments, one can also show that 
$$
\liminf_{N\to\infty}\frac{\int_{\theta^*_2-\delta}^{\theta^*_2+\delta}m_N(\theta)g_{\mu}(\theta)d\theta}{\int_{\theta^*_1-\delta}^{\theta^*_1+\delta}m_N(\theta)g_{\mu}(\theta)d\theta}\geq \frac{g_{\mu}(\xi_2)e^{(d_2-\varepsilon-\delta C)}}{g_{\mu}(\xi_1)e^{(d_1+\varepsilon+\delta C)}}\sqrt{\frac{-2\varepsilon-h''(\theta^*_1)}{2\varepsilon-h''(\theta^*_2)}}.
$$
By taking first $\delta\to 0$ and using the continuity of $g_{\mu}$, and then $\varepsilon\to 0$, the result follows.
\end{proof}

\medskip

\thmdouble*

\begin{proof}[{\bf Proof of Theorem \ref{theo:twodominance}}]

 First notice that by following the proof of Theorem \ref{theo:unique}, one can check that Condition (A) of {Key Lemma~\ref{th:main}} holds with either $i=1$ or $i=2$. 
 
To conclude the proof, it remains to show that Condition (B) of 
{Key Lemma~\ref{th:main}} holds with $c_{1,2}=(-e^{d_2}g_{\mu}(\theta^*_2)h'(\theta^*_1))/(e^{d_1}g_{\mu}(\theta^*_1)h'(\theta^*_2)).$ The proof of this fact is very similar to that of Remark \ref{cor:1d_K2}.

In what follows, for any $\delta>0$, denote $V_1(\delta)=(\theta^*_1,\theta^*_1+\delta)$, $\bar{V}_1(\delta)=[\theta^*_1,\theta^*_1+\delta]$, $V_2(\delta)=(\theta^*_2-\delta,\theta^*_2)$ and $\bar{V}_2(\delta)=[\theta^*_2-\delta,\theta^*_2]$.   

Take $\delta>0$ small enough and apply the Mean Valued Theorem to conclude that 
\begin{equation}
\label{ineq_0_proof_thm_3}
\frac{\int_{\theta^*_2-\delta}^{\theta^*_2}m_N(\theta)g_{\mu}(\theta)d\theta}{\int_{\theta^*_1}^{\theta^*_1+\delta}m_N(\theta)g_{\mu}(\theta)d\theta}=\frac{g_{\mu}(\xi_2)\int_{\theta^*_2-\delta}^{\theta^*_2}e^{Nh_N(\theta)}d\theta}{g_{\mu}(\xi_1)\int_{\theta^*_1}^{\theta^*_1+\delta}e^{Nh_N(\theta)}d\theta}\;,    
\end{equation}
where  $\xi_i\in V_i(\delta)$ for $1\leq i\leq 2$. 

Since $h'(\theta^*_1)<0$ and $h'(\theta^*_2)>0$, one can take $\varepsilon>0$ such that $h'(\theta^*_1)+2\varepsilon<0$ and $h'(\theta^*_2)-2\varepsilon>0$. 
By the continuity of $h'$, one can find $\delta$ sufficiently small such that 
\begin{equation}
\label{ineq_1_proof_thm_3}
 h'(\theta^*_i)-\varepsilon\leq h'(\theta)\leq h'(\theta^*_i)+\varepsilon,   \end{equation}
for all $\theta\in \bar{V}_i(\delta)$ and both $i=1$ and $i=2$.

The uniform convergence of $h'_N$ to $h'$ around $\theta^*_i$ and \eqref{ineq_1_proof_thm_3} ensure that there exists $N_0=N_0(\delta,\varepsilon)$ such that for all $N\geq N_0$, 
\begin{equation}
\label{ineq_2_proof_thm_3}
h'(\theta^*_i)-2\varepsilon \leq h'_N(\theta)\leq h'(\theta^*_i)+2\varepsilon,    
\end{equation}
for all $\theta\in \bar{V}_i(\delta)$ and $1\leq i\leq 2$.

Now, since $h_N\in C^1([a,b])$, for each $1\leq i\leq 2$, we can expand $h_N$ around $\theta^*_i$ up to order 1 to deduce that 
for all $\theta\in \bar{V}_i(\delta)$, 
$$
h_N(\theta)=h_N(\theta^*_i)+h'_N(\zeta_{i,\theta})(\theta-\theta^*_i)\;,
$$
where $\zeta_{i,\theta}$ is between $\theta$ and $\theta^*_i$.

It follows from \eqref{ineq_1_proof_thm_3} and \eqref{ineq_2_proof_thm_3} that for all $\theta\in\bar{V}_1(\delta)$ and $N\geq N_0$, it holds
\begin{align}
\label{ineq_3_proof_thm_3}
N(h'(\theta^*_1)-2\varepsilon)(\theta-\theta^*_1)\leq N(h_N(\theta)-h_N(\theta^*_1))
\leq N(h'(\theta^*_1)+2\varepsilon)(\theta-\theta^*_1).
\end{align} 
Similarly, for all $\theta\in\bar{V}_2(\delta)$ and $N\geq N_0$, it holds that
\begin{align}
\label{ineq_4_proof_thm_3}
N(h'(\theta^*_2)+2\varepsilon)(\theta-\theta^*_2)\leq N(h_N(\theta)-h_N(\theta^*_2))
\leq N(h'(\theta^*_2)-2\varepsilon)(\theta-\theta^*_2).
\end{align} 

By using the fact that $N(h_N(\theta^*_i)-h(\theta_i^*))\to d_i$ together with  inequalities \eqref{ineq_3_proof_thm_3} and \eqref{ineq_4_proof_thm_3}, we deduce that for all $N\geq N_0$ and any $\theta\in\bar{V_1}(\delta)$,
\begin{align}
\label{ineq_5_proof_thm_3}
d_1-\varepsilon+N(h'(\theta^*_1)-2\varepsilon)(\theta-\theta^*_1)&\leq N(h_N(\theta)-h(\theta^*_1)) \nonumber \\
&\leq d_1+\varepsilon +N(h'(\theta^*_1)+2\varepsilon)(\theta-\theta^*_1).
\end{align} 
and similarly, for all $N\geq N_0$ and any $\theta\in\bar{V_2}(\delta),$
\begin{align}
\label{ineq_6_proof_thm_3}
d_2-\varepsilon+N(h'(\theta^*_2)+2\varepsilon)(\theta-\theta^*_2)&\leq N(h_N(\theta)-h(\theta^*_2)) \nonumber \\
&\leq d_2+\varepsilon +N(h'(\theta^*_2)-2\varepsilon)(\theta-\theta^*_2).
\end{align}

Therefore, we deduce from the inequalities above that for all $N\geq N_0$ (recall that $m_N(\theta)=e^{Nh_N(\theta)}$ and that $h(\theta^*_1)=h(\theta^*_2)$ by assumption), it holds
\begin{align}\label{eq:*_thm_3}
\frac{\int_{\theta^*_2-\delta}^{\theta^*_2}m_N(\theta)d\theta}{\int_{\theta^*_1}^{\theta^*_1+\delta}m_N(\theta)d\theta}&\leq \frac{e^{(d_2+\varepsilon)}\int_{\theta^*_2-\delta}^{\theta^*_2}e^{N(h'(\theta^*_2)-2\varepsilon)(\theta-\theta^*_2)}d\theta}{e^{(d_1-\varepsilon)}\int_{\theta^*_1}^{\theta^*_1+\delta}e^{N(h'(\theta^*_1)-2\varepsilon)(\theta-\theta^*_1)}d\theta}\;.
\end{align}
By making the change of variables $x=-N(\theta-\theta^*_2)$ (respectively $x=N(\theta-\theta^*_1)$) in the integral appearing in the numerator (respectively denominator) of the right-most ratio, we then obtain
\begin{equation}
\label{ineq_7_proof_thm_3}
\frac{\int_{\theta^*_2-\delta}^{\theta^*_2}e^{N(h'(\theta^*_2)-2\varepsilon)(\theta-\theta^*_2)}d\theta}{\int_{\theta^*_1}^{\theta^*_1+\delta}e^{N(h'(\theta^*_1)-2\varepsilon)(\theta-\theta^*_1)}d\theta}=\frac{\int_{0}^{\delta N}e^{-(h'(\theta^*_2)-2\varepsilon)x }dx}{\int_{0}^{\delta N}e^{-(2\varepsilon-h'(\theta^*_1))x}dx}\;.    
\end{equation}

Therefore, from \eqref{ineq_0_proof_thm_3}, \eqref{eq:*_thm_3} and \eqref{ineq_7_proof_thm_3}, we have that for all $N\geq N_0$,
$$\frac{\int_{\theta^*_2-\delta}^{\theta^*_2}m_N(\theta)g_{\mu}(\theta)d\theta}{\int_{\theta^*_1}^{\theta^*_1+\delta}m_N(\theta)g_{\mu}(\theta)d\theta}\leq \frac{g_{\mu}(\xi_2)e^{(d_2+\varepsilon)}\int_{0}^{\delta N}e^{-(h'(\theta^*_2)-2\varepsilon)x}dx}{g_{\mu}(\xi_1)e^{(d_1-\varepsilon)}\int_{0}^{\delta N}e^{-(2\varepsilon-h'(\theta^*_1))x}dx}\;,
$$
where $\xi_1\in V_i(\delta)$ for $1\leq i\leq 2.$
As a consequence, it follows that 
$$
\limsup_{N\to\infty}\frac{\int_{\theta^*_2-\delta}^{\theta^*_2+\delta}m_N(\theta)g_{\mu}(\theta)d\theta}{\int_{\theta^*_1-\delta}^{\theta^*_1+\delta}m_N(\theta)g_{\mu}(\theta)d\theta}\leq \frac{g_{\mu}(\xi_2)e^{(d_2+\varepsilon)}}{g_{\mu}(\xi_1)e^{(d_1-\varepsilon)}}\frac{2\varepsilon-h'(\theta^*_1)}{h'(\theta^*_2)-2\varepsilon}.
$$
By similar arguments, one can also deduce that 
$$
\liminf_{N\to\infty}\frac{\int_{\theta^*_2-\delta}^{\theta^*_2}m_N(\theta)g_{\mu}(\theta)d\theta}{\int_{\theta^*_1}^{\theta^*_1+\delta}m_N(\theta)g_{\mu}(\theta)d\theta}\geq \frac{g_{\mu}(\xi_2)e^{(d_2-\varepsilon)}}{g_{\mu}(\xi_1)e^{(d_1+\varepsilon)}}\frac{(-2\varepsilon-h'(\theta^*_1))}{2\varepsilon+h'(\theta^*_2)}.
$$
By taking first $\delta\to 0$ and using the continuity of $g_{\mu}$, and then $\varepsilon\to 0$, we conclude the proof.
\end{proof}

\subsection{Expectation of hitting times}\label{sec:proofs_expectations}

 We give below the computations of the expected hitting  time  for the Single-well and Alternating-wells random walk presented in Section~\ref{sec:sw}  and Section~\ref{sec:aw}, respectively.

\medskip
\propSW*

\begin{proof}[{\bf Proof of Proposition~\ref{Prop:exp_of_tau_SW}}]
We have $\mathcal X_N=[-N,N]\cap \bZ$ and let $\tau_{\{-N,N\}}$ the time to hit the target set $\{-N,N\}$. 
We shall prove only the case  $\theta\neq \{1/2,{\color{blue} 1}\}$. The case $\theta=1/2$ is the classical Gambler's ruin on $[-N,N]\cap \bZ$ starting from $0$ and the case $\theta=1$ is immediate. 
%
%
Also, given a set $A$, we denote by  $\tau_{A}$ the corresponding hitting time.

By the Strong Markov property and the symmetry, we have
\begin{align*}
\bE_{0, \delta_\theta}\left[\tau_{\{-N,N\}}\right]&=1+\theta\bE_{-1, \delta_\theta}\left[\tau_{\{-N,N\}}\right]+(1-\theta)\bE_{1,\delta_\theta}\left[\tau_{\{-N,N\}}\right]\\
&=1+\bE_{1,\delta_\theta}\left[\tau_{\{-N,N\}}\right]\;.
\end{align*}
On the other hand, by the strong Markov property, we also get
\begin{align*}
\bE_{1,\delta_\theta}\left[\tau_{\{-N,N\}}\right]=\bE_{1,\delta_\theta}\left[\tau_{\{0,N\}}\right]+\bP_{1,\delta_\theta}(\tau_{\{0\}}<\tau_{\{N\}})\bE_{0,\delta_\theta}\left[\tau_{\{-N,N\}}\right]\;,
\end{align*}
and thus
\begin{align*}
\bE_{0,\delta_\theta}\left[\tau_{\{-N,N\}}\right]&=1+\bE_{1,\delta_\theta}\left[\tau_{\{0,N\}}\right]
+\bP_{1,\delta_\theta}(\tau_{\{0\}}<\tau_{\{N\}})\bE_{0,\delta_\theta}\left[\tau_{\{-N,N\}}\right]\;.
\end{align*}
Solving the above equation for $\bE_{0,\delta_\theta}\left[\tau_{\{-N,N\}}\right]$, we obtain that
\begin{equation}
\label{meain_eq_for_Exp_tau_N_p}
\bE_{0,\delta_\theta}\left[\tau_{\{-N,N\}}\right]=\frac{1+\bE_{1,\delta_\theta}\left[\tau_{\{0,N\}}\right]}{\bP_{1,\delta_\theta}(\tau_{\{0\}}\geq \tau_{\{N\}})}\;.
\end{equation}

Finally, since the dynamics of the Single-well random walk on $\{1,\ldots, N\}$ is the same as the classical Gambler's Ruin random walk on $\{0,\ldots, N\}$, we obtain  that
$$
\bE_{1,\delta_\theta}\left[\tau_{\{0,N\}}\right]=\frac{1}{1-2\theta}-\frac{N}{\theta}\frac{1}{(((1-\theta)/\theta)^N-1)}\;,
$$
and that
$$
\bP_{1,\delta_\theta}(\tau_{\{0\}}>\tau_{\{N\}})=\frac{1-2\theta}{\theta(((1-\theta)/\theta)^N-1)}\;.
$$
By plugging these expressions in \eqref{meain_eq_for_Exp_tau_N_p}, the result follows.

\end{proof}


\medskip
\propAW*

\begin{proof}[{\bf Proof of Proposition~\ref{Prop:esp__of_tau_alternating_wells}}]

To facilitate the exposition,  in this  proof we shall carry the dependence on the state $\theta$ and an initial state $i$ in the random variable $\tau$, rather than in the probability/expectation. Thus,  given $i\in \mathcal X_N$, and a set  $A\subset \mathcal X_N$, we denote by  $\tau^{i,\theta}_{A}$ the first time the random walk  with state $\theta$ hits the set $A$, starting from the state $i$, and we write $\bE [\tau^{i,\theta}_{A}]$ and  $\bP (\tau^{i,\theta}_{A}<\tau^{i,\theta}_{B})$ instead of $\bE_{i, \delta_\theta}[\tau_A]$ and   $\bP_{i,\delta_\theta} (\tau_{A}<\tau_{B})$.
Moreover, since the proof also relies upon comparison across  different models, we shall upper indexing $\bP$ and $\bE$ by $\gr$ (gambler ruin), $\sw$ (single-well) and $\aw$ (alternating-wells), to emphasize which model we are referring to.

We have $\mathcal X_N=[-2N,2N]\cap \bZ$ and we need to compute  $\bE^{\aw}[\tau^{0,\theta}_{\{-2N,2N\}}]$.  
The random variable $\tau^{0,\theta}_{\{-2N,2N\}}$ can be decomposed as follows:
\begin{align*}
    \tau^{0,\theta}_{\{-2N,2N\}} &=
    \big( \tau^{0,\theta}_{\{N\}} + T_1 \big)\1_{(\tau^{0,\theta}_{\{N\}}\leq \tau^{0,\theta}_{\{-N\}})} + \big( \tau^{0,\theta}_{\{-N\}} + T_2 \big)\1_{(\tau^{0,\theta}_{\{-N\}}<\tau^{0,\theta}_{\{N\}})}
    \\
    &=  \tau^{0,\theta}_{\{-N,N\}}  + T_1  \1_{(\tau^{0,\theta}_{\{N\}}\leq \tau^{0,\theta}_{\{-N\}})} + T_2 \1_{(\tau^{0,\theta}_{\{-N\}}<\tau^{0,\theta}_{\{N\}})}\;,
\end{align*}
where $T_1$ and $T_2$ are random variables  independent of everything else and have the same distribution as $\tau^{N,\theta}_{\{-2N,2N\}}$ and $\tau^{-N,\theta}_{\{-2N,2N\}}$, respectively. These random variables, in turn, can be decomposed as:
\begin{align*}
    T_1 &= 
    \tau^{N,\theta}_{\{0,2N\}} + \tau' \1_{(\tau^{N,\theta}_{\{0\}})\leq \tau^{N,\theta}_{\{2N\}})}\;,
    \\
    T_2 &=
    \tau^{-N,\theta}_{\{0,-2N\}} + \tau'' \1_{(\tau^{-N,\theta}_{\{0\}})\leq \tau^{-N,\theta}_{\{-2N\}})}\;,
\end{align*}
with  $\tau', \tau''$  random variables with the same distribution as $\tau^{0,\theta}_{\{-2N, 2N\}}$,  independent of everything else. 
Overall,  we can write $\tau^{0,\theta}_{\{-2N,2N\}}$ as  
\begin{align*}
   \tau^{0,\theta}_{\{-2N, 2N\}} &= \tau^{0,\theta}_{\{-N, N\}}  + \tau^{N,\theta}_{\{0, 2N\}}  \1_{(\tau^{0,\theta}_{\{N\}} \leq \tau^{0,\theta}_{\{-N\}})}
 +  \tau^{-N,\theta}_{\{0, -2N\}}
    \1_{(\tau^{0,\theta}_{\{-N\}}< \tau^{0,\theta}_{\{N\}})}
    \\ \nonumber
    &
    +\tau' \1_{(\tau^{0,\theta}_{\{N\}}\leq \tau^{0,\theta}_{\{-N\}})} \1_{(\tau^{N,\theta}_{\{0\}}< \tau^{N,\theta}_{\{2N\}})} 
    + \tau'' \1_{(\tau^{0,\theta}_{\{-N\}}< \tau^{0,\theta}_{\{N\}})} \1_{(\tau^{-N,\theta}_{\{0\}}\leq \tau^{-N,\theta}_{\{-2N\}})} 
    \;.
\end{align*}

Using independence between the  random variables,  we obtain that
\begin{align}
\label{eq:decomposition}
\bE^{\aw}(\tau^{0,\theta}_{\{-2N, 2N\}})&= \bE^{\aw}(\tau^{0,\theta}_{\{-N, N\}}) +
\bE^{\aw}(\tau^{N,\theta}_{\{0, 2N\}})
\bP^{\aw}(\tau^{0,\theta}_{\{N\}} \leq \tau^{0,\theta}_{\{-N\}})
\\ \nonumber 
& + \bE^{\aw}(\tau^{-N,\theta}_{\{0, -2N\}})
\bP^{\aw}(\tau^{0,\theta}_{\{-N\}}< \tau^{0,\theta}_{\{N\}})
\\ \nonumber 
&+ \bE^{\aw}(\tau^{0,\theta}_{\{-2N, 2N\}}) 
\bP^{\aw}(\tau^{0,\theta}_{\{N\}} \leq \tau^{0,\theta}_{\{-N\}})\bP^{\aw}(\tau^{N,\theta}_{\{0\}} < \tau^{N,\theta}_{\{2N\}}) \\ \nonumber 
&+ \bE^{\aw}(\tau^{0,\theta}_{\{-2N, 2N\}}) 
\bP^{\aw}(\tau^{0,\theta}_{\{-N\}}< \tau^{0,\theta}_{\{N\}})\bP^{\aw}(\tau^{-N,\theta}_{\{0\}} \leq  \tau^{-N,\theta}_{\{-2N\}})\;.
\end{align}

Using the Strong Markov property, we have that:

\begin{itemize}
    \item  $\bE^{\aw}(\tau^{N,\theta}_{\{0, 2N\}})= \bE^{\sw}(\tau^{0,\theta}_{\{-N, N\}})$, where the right-hand side (rhs) is  the expected hitting time of $\{-N, N\}$  for a \sw{} random  walk in a $\theta$-environment starting from $0$  (see, Proposition~\ref{Prop:exp_of_tau_SW});
    \item $\bE^{\aw}(\tau^{-N,\theta}_{\{0, -2N\}})= \bE^{\sw}(\tau^{0,1-\theta}_{\{-N, N\}})$, where the rhs is the  expected  hitting time of $\{-N, N\}$  for a \sw{} random  walk in a $(1-\theta)$-environment starting from $0$; 
    \item $\bE^{\aw}(\tau^{0,\theta}_{\{-N, N\}})= \bE^{\gr}(\tau^{0,1-\theta}_{\{-N, N\}})$, where the rhs is the expected hitting time of $\{-N,N\}$ for a \gr{} random walk starting from $0$ (when it jumps to the right with probability $1-\theta$; note that by symmetry $\bE^{\gr}(\tau^{0,1-\theta}_{\{-N, N\}})=\bE^{\gr}(\tau^{0,\theta}_{\{-N, N\}})$); 
    \item $\bP^{\aw}(\tau^{0,\theta}_{\{N\}} \leq \tau^{0,\theta}_{\{-N\}})=
    \bP^{\gr}(\tau^{0,1-\theta}_{\{N\}} \leq \tau^{0,1-\theta}_{\{-N\}})=\frac{1}{1+\left(\frac{\theta}{1-\theta} \right)^{N}}
    $, where the rhs is  the probability of  a Gambler's ruin hitting  $N$ before $-N$ starting from $0$ (when it jumps to the right with probability $1-\theta$);
    \item $\bP^{\aw}(\tau^{-N,\theta}_{\{0\}} \leq  \tau^{-N,\theta}_{\{-2N\}})
     = \bP^{\sw}(\tau^{0,1-\theta}_{\{N\}} \leq   \tau^{0,1-\theta}_{\{-N\}})
    = 1-\theta$, which can be proved by a standard first-step analysis.
\end{itemize}

Therefore, from~\eqref{eq:decomposition}  we then obtain
\begin{align*}
&\bE^{\aw}(\tau^{0,\theta}_{\{-2N, 2N\}})
=\\
&\frac{1}{\frac{\theta \left(\frac{\theta}{1-\theta} \right)^{N} + 1-\theta}{1+\left(\frac{\theta}{1-\theta} \right)^{N}} }\left( \bE^{\gr}(\tau^{0,1-\theta}_{\{-N, N\}})  + \frac{\bE^{\sw}(\tau^{0,\theta}_{\{-N, N\}})}{1+\left(\frac{\theta}{1-\theta} \right)^{N}}  + \frac{\bE^{\sw}(\tau^{0,1-\theta}_{\{-N, N\}})\left(\frac{\theta}{1-\theta} \right)^{N}}{1+\left(\frac{\theta}{1-\theta} \right)^{N}}\right)\;.
\end{align*}
Using Proposition~\ref{Prop:exp_of_tau_SW} and~\eqref{eq:timeGR}, the claim follows.

\end{proof}

\subsection{Metastability/cut-off}\label{sec:proofs_unpredict}

In this section, we prove Theorem~\ref{thm:metastabilitySW}, Proposition \ref{thm:metastability_GR} and Theorem~\ref{th:meta_aw}.




\medskip
\thmmetaSW*

\begin{proof}[{\bf Proof of Theorem~\ref{thm:metastabilitySW}}]
Let us start proving item (i). Given a subset $A\subset \mathcal X_N$, let us denote $\tau_{A}$ the first time the single-well random walk  hits the set $A$. 
With this notation, we have  $\tau_{N}=\tau_{\{-N,N\}}$.  Since the proof  relies upon comparison with the gambler ruin, we shall upper indexing $\bP$ and $\bE$ by $\gr$ (gambler ruin) and  $\sw$ (single-well) to emphasize which model we are referring to.

According to Theorem \ref{thm:fernandezetal}, to show item $i)$ it suffices to find $R_N>0$ and $r_N\in (0,1)$ such that $r_N \underset{N\to\infty}{\longrightarrow} 0$, $R_N/\bE^{\sw}_{0,\delta_\theta}[\tau_{\{-N,N\}}]\underset{N\to\infty}{\longrightarrow} 0$ and
\begin{equation*}
\sup_{x\in \{-N+1,\ldots,-1,1,\ldots, N-1\}}\bP^{\sw}_{x, \delta_\theta}(\tau_{\{-N,0,N\}}>R_N)\leq r_N\;.    
\end{equation*}
By the dynamics of the single-well random walk, for each $i\in\{1,\ldots, N-1\}:$
$$
\bP^{\sw}_{i,\delta_\theta}(\tau_{\{-N,0,N\}}>R_N)=\bP^{\gr}_{i,\theta}(\tau_{\{0,N\}}>R_N)\;,
$$
where, $\bP^{\gr}_{i,\theta}$ denotes the probability of a random walk starting from $i$ with transition probability to the right equal to $\theta$ (we shall denote by $\bE^{\gr}_{i,\theta}$ the corresponding expectation). Similarly, we also have that for each $i\in\{-1,\ldots, -N+1\}:$ 
$$
\bP^{\sw}_{i, \delta_\theta}(\tau_{\{-N,0,N\}}>R_N)=\bP^{\gr}_{i,\theta}(\tau_{\{-N,0\}}>R_N)=\bP^{\gr}_{-i,\theta}(\tau_{\{0,N\}}>R_N)\;.
$$
Hence, by Markov's inequality,
$$
\sup_{x\in \{-N+1,\ldots,-1,1,\ldots, N-1\}}\bP^{\sw}_{x,\delta_\theta}(\tau_{\{-N,0,N\}}>R_N)\leq \frac{1}{R_N}\max_{i\in \{1,\ldots, N-1\}} \bE^{\gr}_{i,\theta}(\tau_{\{0,N\}})\;.
$$
It is well know that for any $i\in \{1,\ldots, N-1\},$
$$
\bE^{\gr}_{i,\theta}[\tau_{\{0,N\}}]=\frac{i}{1-2\theta} -\frac{N}{1-2\theta}\frac{\left(\frac{1-\theta}{\theta}\right)^i -1}{\left(\frac{1-\theta}{\theta}\right)^N -1}\;.
$$
Thus, for all $\theta<1/2$ and $i\in \{1,\ldots, N-1\},$
$$
\bE^{\gr}_{i,\theta}[\tau_{\{0,N\}}]\leq \frac{i}{1-2\theta}\leq \frac{N-1}{1-2\theta}\;, 
$$
so that
$$
\sup_{x\in \{-N+1,\ldots,-1,1,\ldots, N-1\}}\bP^{\sw}_{x,\delta_\theta}(\tau_{\{-N,0,N\}}>R_N)\leq \frac{N-1}{R_N(1-2\theta)}\;.
$$
Therefore, by taking $R_N=((N-1)/(1-2\theta))^{1+\varepsilon}$ for any fixed $\varepsilon>0$  and $r_N=\frac{N-1}{R_N(1-2\theta)}=((1-2\theta)/(N-1))^{\varepsilon}$, and recalling that by Proposition \ref{Prop:exp_of_tau_SW} $\bE^{\sw}_{0,\delta_\theta}[\tau_{\{-N,N\}}]$ grows exponentially in $N$ for any $0<\theta<1/2$, the result follows.

We now prove item (ii).  Let $\{X_n\}_{n\ge0}$ be the location chain with state $\theta$, that is, the Markov chain started at $0$ with transition matrix $Q^{(\theta)}$ defined in Section~\ref{sec:sw}. Notice that $\{|X_n|\}_{n\ge0}$ is a simple random walk reflected at $0$, with probability $1$ to go from $0$ to $1$ and probability $\theta$ to go from $i$ to $i+1$ for any $i\ge1$. Since here we consider $\theta>1/2$, the strong law of large numbers gives
\[
\frac{|X_{n}|}{n}\underset{n\rightarrow\infty}{\rightarrow} 2\theta-1\,,\,\,\text{a.s.}
\]
Since $\tau_N\stackrel{N\rightarrow\infty}{\rightarrow}\infty$ a.s., it follows that , $\frac{|X_{\tau_N}|}{\tau_N}\rightarrow 2\theta-1$ a.s. as well. Moreover, by definition, $|X_{\tau_N}|=N,N\ge1$. On the other hand, since $\theta>1/2$, by Proposition \ref{Prop:exp_of_tau_SW}, we have $\frac{\bE^{\sw}_{0,\delta_\theta}[\tau_{N}]}{N}\rightarrow \frac{1}{2\theta-1}$. Thus
\[
\frac{\tau^N}{\bE^{\sw}_{0,\delta_\theta}[\tau_{N}]}=\frac{\tau^N}{N}\frac{N}{\bE^{\sw}_{0,\delta_\theta}[\tau_{N}]}=\frac{\tau^N}{|X_{\tau_N}|}\frac{N}{\bE^{\sw}_{0,\delta_\theta}[\tau_{N}]}\rightarrow1\,,\,\,\text{a.s.}
\]
\end{proof}

\medskip
\propmetaGR*

\begin{proof}[{\bf Proof of Proposition \ref{thm:metastability_GR}}]
The proof follows the same lines as the proof of Item ii) of Theorem ~\ref{thm:metastabilitySW} with very minor changes.
\end{proof}

\medskip
\thmmetaAW*

\begin{proof}[{\bf Proof of Theorem \ref{th:meta_aw}}]
Given a subset $A\subset \mathcal X_N=[-2N,2N]\cap \mathbb{Z}$, let us denote $\tau_{A}$ the first time the alternating-wells random walk  hits the set $A$. 
With this notation, we have  $\tau_{N}=\tau_{\{-2N,2N\}}$.  Since the proof  relies upon comparison with the single-well, we shall upper indexing $\bP$ and $\bE$ by $\sw$ (single-well) and $\aw$ (alternating-wells) to emphasize which model we are referring to.

According to Theorem~\ref{thm:fernandezetal}, it suffices to find $R_N>0$ and $r_N\in (0,1)$ such that $r_N \underset{N\to\infty}{\longrightarrow} 0$, $R_N/\bE^{\aw}_{0, \delta_\theta}[\tau_{\{-2N,2N\}}]\underset{N\to\infty}{\longrightarrow} 0$ and
\begin{equation*}
\sup_{x\in \{-2N+1,\ldots,-N-1,-N+1,\ldots, 2N-1\}}\bP^{\aw}_{x,\delta_\theta}(\tau_{\{-2N,0,2N\}}>R_N)\leq r_N\;.
\end{equation*}

By Markov's inequality, the left-hand side above is less than or equal to $\frac{1}{R_N} \max_{x\in \mathcal{X}_N}\bE^{\aw}_{x,\delta_\theta}[\tau_{\{-2N,0,2N\}}]$. 
The following claim holds:

{\bf Claim:} 
\begin{align*}
\max_{x\in \mathcal{X}_N}\bE^{\aw}_{x,\delta_\theta}[\tau_{\{-2N,0,2N\}}] = \begin{cases}
\bE^{\aw}_{-N,\delta_\theta}[\tau_{\{-2N,0,2N\}}]\;, & \text{ if $\theta>1/2$,}
\\[5pt]
\bE^{\aw}_{N,\delta_\theta}[\tau_{\{-2N,0,2N\}}]\;, &  \text{ if $\theta<1/2$.}
\end{cases}
\end{align*}

Note that $\bE^{\aw}_{-N,\delta_\theta}[\tau_{\{-2N,0,2N\}}]= \bE^{\aw}_{-N,\delta_\theta}[\tau_{\{-2N,0\}}]$ and  $\bE^{\aw}_{N,\delta_\theta}[\tau_{\{-2N,0,2N\}}]=\bE^{\aw}_{N,\delta_\theta}[\tau_{\{0,2N\}}]$. 
By comparison with the single-well random walk  in Section~\ref{sec:sw}, we obtain that 
\begin{equation}\label{eq:totti}
\begin{split}
\bE^{\aw}_{N,\delta_\theta}[\tau_{\{0,2N\}}]&= \bE^{\sw}_{0,\delta_\theta}[\tau_{\{-N,N\}}]\;,
\\
\bE^{\aw}_{-N,\delta_\theta}[\tau_{\{-2N,0\}}]&= \bE^{\sw}_{0,\delta_{(1-\theta)}}[\tau_{\{-N,N\}}]\;.
\end{split}
\end{equation}
From Proposition~\ref{Prop:exp_of_tau_SW},  we have that 
\[
\begin{cases}
\bE^{\sw}_{0,\delta_\theta}[\tau_{\{-N,N\}}] \leq \frac{N}{2\theta -1}\;, & \text{ if $\theta>1/2$,}
\\[5pt]
\bE^{\sw}_{0,\delta_{(1-\theta)}}[\tau_{\{-N,N\}}] \leq \frac{N}{1-2\theta}\;, & \text{ if $\theta<1/2$.}
\end{cases} 
\]
Thus, for every $\theta\neq 1/2$ we obtain that  
$$
\sup_{x\in \mathcal{X}_N }\bP^{\aw}_{x,\delta_\theta}(\tau_{\{-N,0,N\}}>R_N)\leq \frac{N}{R_N|2\theta-1|}\;.
$$
Therefore, by taking $R_N=(N/|2\theta-1|)^{1+\varepsilon}$ for any fixed $\varepsilon>0$  and $r_N=\frac{N}{R_N|2\theta-1|}=(|2\theta-1|/N)^{\varepsilon}$, and recalling that by Proposition~\ref{Prop:esp__of_tau_alternating_wells} $\bE^{\aw}_{0,\delta_\theta}[\tau_{\{-2N,2N\}}]$ grows exponentially in $N$ for any $\theta\neq 1/2$, the result follows.    

\medskip 
{\em Proof of the Claim:} 
Let us begin observing that for every $x \in [-2N,2N]\cap \mathbb{Z}$ we have
\[
\bE^{\aw}_{x,\delta_\theta}[\tau_{\{-2N,0,2N\}}]=\begin{cases}
\bE^{\aw}_{x,\delta_\theta}[\tau_{\{-2N,0\}}]=\bE^{\sw}_{x+N,\delta_{(1-\theta)}}[\tau_{\{-N,N\}}]\;, & x<0\;,
\\
\bE^{\aw}_{x,\delta_\theta}[\tau_{\{0,2N\}}]=\bE^{\sw}_{x-N,\delta_{\theta}}[\tau_{\{-N,N\}}]\;, & x>0\;,
\\
0\;, &  x=0\;,
\end{cases}
\]
where, we used a comparison with the single-well random walk (see,  Section~\ref{sec:sw}). 
Thus, $\max_{x\in \mathcal{X}_N}\bE^{\aw}_{x,\delta_\theta}[\tau_{\{-2N,0,2N\}}]$ can be written as 
\[
 \max \left\{ \max_{\substack{x\in \mathcal{X}_N \\x<0}}\bE^{\sw}_{x+N,\delta_{(1-\theta)}}[\tau_{\{-N,N\}}], 
\max_{\substack{x\in \mathcal{X}_N \\x>0}}\bE^{\sw}_{x-N,\delta_{\theta}}[\tau_{\{-N,N\}}] 
\right\}\;.
\]
Hence, to prove the claim it suffices to show that for every $\theta \in (0,1]$
\begin{align}\label{eq:maxsw}
\max_{i\in \{-N+1,\ldots,-1,0,+1,\ldots, N-1\}}\bE^{\sw}_{i,\delta_{\theta}}[\tau_{\{-N,N\}}]= \bE^{\sw}_{0,\delta_{\theta}}[\tau_{\{-N,N\}}] \;.
\end{align}
As a matter of fact, if \eqref{eq:maxsw} holds,  we obtain $$\max_{x\in \mathcal{X}_N}\bE^{\aw}_{x,\delta_\theta}[\tau_{\{-2N,0,2N\}}] = \max \left\{ \bE^{\sw}_{0,\delta_{(1-\theta)}}[\tau_{\{-N,N\}}], \bE^{\sw}_{0,\delta_{\theta}}[\tau_{\{-N,N\}}] 
\right\}$$ and the claim follows from Proposition~\ref{Prop:exp_of_tau_SW} and \eqref{eq:totti}. To show that  \eqref{eq:maxsw} holds, note that, for every $\theta \in (0,1]$ and  for every $x \in \{1,\ldots, N-1\}$,  it holds that $\bE^{\sw}_{x,\delta_{\theta}}[\tau_{\{-N,N\}}]=\bE^{\sw}_{-x,\delta_{\theta}}[\tau_{\{-N,N\}}]$ (due to  the symmetry of the single-well). Therefore, showing  \eqref{eq:maxsw} reduces to show $
\max_{i\in \{0,1,\ldots, N-1\}}\bE^{\sw}_{i,\delta_{\theta}}[\tau_{\{-N,N\}}]= \bE^{\sw}_{0,\delta_{\theta}}[\tau_{\{-N,N\}}]$. The latter condition follows from noticing that $\bE^{\sw}_{x,\delta_{\theta}}[\tau_{\{-N,N\}}]\leq \bE^{\sw}_{x-1,\delta_{\theta}}[\tau_{\{-N,N\}}]$, for every $x\in \{1, \ldots, N-1\}$. In fact, $\bE^{\sw}_{0,\delta_{\theta}}[\tau_{\{-N,N\}}]= 1 + \bE^{\sw}_{1,\delta_{\theta}}[\tau_{\{-N,N\}}]$ and thus $\bE^{\sw}_{1,\delta_{\theta}}[\tau_{\{-N,N\}}]\leq \bE^{\sw}_{0,\delta_{\theta}}[\tau_{\{-N,N\}}]$. Then, recursively applying   that, for every $x\in \{1, \ldots, N-1\}$ it holds that 
\[
\bE^{\sw}_{x,\delta_{\theta}}[\tau_{\{-N,N\}}]= 1 + (1-\theta)\bE^{\sw}_{x-1,\delta_{\theta}}[\tau_{\{-N,N\}}] + \theta \bE^{\sw}_{x+1,\delta_{\theta}}[\tau_{\{-N,N\}}]\;,
\]
and   $\bE^{\sw}_{x,\delta_{\theta}}[\tau_{\{-N,N\}}]\leq \bE^{\sw}_{x-1,\delta_{\theta}}[\tau_{\{-N,N\}}]$, the claim follows. 
\end{proof}

\section{Auxiliary Lemmas}\label{app:lemmas}

\begin{lemma}
\label{Lemma:SW_converges_uniformly}
Let $h_N(\theta)=N^{-1}\log m_N(\theta)$ be the Single Well potential with parameter $N\geq 1$ at the value $\theta\in (0,1]$, where $m_N(\theta)$ is given in Equation~\eqref{Expression_for_exp_of_tau_SW}, and $h(\theta)$ be the function defined in Definition~\eqref{def:limit_potential_SW}. For any $a\in (0,1)$, it holds that
$$
\sup_{\theta\in [a,1]}|h_N(\theta)-h(\theta)|\underset{N\to\infty}{\longrightarrow}0\;.
$$
\end{lemma}
\begin{proof}[Proof of Lemma~\ref{Lemma:SW_converges_uniformly}]
Fix $\varepsilon>0$. The continuity of the function $h$ implies that $h$ is uniformly continuous on $[a,1]$. From this it follows that there exists $\delta=\delta(\varepsilon)$ such that if $\theta,\theta'\in [a,1]$ satisfies $|\theta-\theta'|<\delta$, then $|h(\theta)-h(\theta')|<\varepsilon$. For fixed $\delta>0$, define $M=\lfloor (1-a)/\delta \rfloor$, $\theta_i=a+i\delta$ for $0\leq i\leq M$ and $\theta_{M+1}=1\vee \theta_M$. 

We recall that $h_N$ converges pointwise to $h$ and that all these functions are monotonically decreasing (see the discussion below Proposition \ref{Prop:exp_of_tau_SW}). It follows from the convergence that there exists $N_0=N_0(\delta)$ such that for all $N\geq N_0$ and $0\leq i\leq M+1$,
$$
|h_N(\theta_i)-h(\theta_i)|<\varepsilon\;.
$$
Note that by construction $|h(\theta_i)-h(\theta)|<\varepsilon$ and $|h(\theta_{i+1})-h(\theta)|<\varepsilon$ for all $\theta_i< \theta<\theta_{i+1}$. On the other hand, using now the monotonicity of both $h_N$ and $h$, we deduce that for all $N\geq N_0$ and $\theta_i< \theta<\theta_{i+1}$,
\begin{equation}
\label{Lemma_1:ineq_1}
h_N(\theta)\leq h_N(\theta_i)\leq {h}(\theta_i)+\varepsilon\leq {h}(\theta)+2\varepsilon \;. 
\end{equation}
Similarly, we also have
\begin{equation}
\label{Lemma_1:ineq_2}
h_N(\theta)\geq h_N(\theta_{i+1})\geq {h}(\theta_{i+1})-\varepsilon\geq {h}(\theta)-2\varepsilon\;. 
\end{equation}
Thus, by combining   \eqref{Lemma_1:ineq_1} and \eqref{Lemma_1:ineq_2} we conclude that for all $N\geq N_0$ and all $0\leq i\leq M+1$, 
$$
\sup_{\theta\in [\theta_i,\theta_{i+1}]}|h_N(\theta)-h(\theta)|<2\varepsilon\;,
$$
implying the lemma.
\end{proof}

\begin{lemma}
\label{Lemma:AW_converges_uniformly}
Let $h_N(\theta)=N^{-1}\log m_N(\theta)$ be the Alternating Wells potential with parameter $N\geq 1$ at the value $\theta\in (0,1)$, where $m_N(\theta)$ is given in Proposition~\ref{Prop:esp__of_tau_alternating_wells}, and $h(\theta)$ be the function defined in  \eqref{def:limit_potential_AW}. For any $0<a<1/2<b<1$, it holds that
$$
\sup_{\theta\in [a,b]}|h_N(\theta)-h(\theta)|\underset{N\to \infty}{\longrightarrow} 0\;.
$$
\end{lemma}
\begin{proof}[Proof of Lemma~\ref{Lemma:AW_converges_uniformly}]
First, apply the exact same proof of Lemma~\ref{Lemma:SW_converges_uniformly} to show the uniform convergence holds on $[a,1/2]$. 
To show the uniform convergence on $[1/2,b]$, one can proceed as in the proof of Lemma~\ref{Lemma:SW_converges_uniformly} but now using $\theta_{i+1}$ instead of $\theta_{i}$ in \eqref{Lemma_1:ineq_1} and $\theta_i$ instead of $\theta_{i+1}$ in \eqref{Lemma_1:ineq_2}.  
\end{proof}

\bibliography{ref}

\begin{thebibliography}{10}
\expandafter\ifx\csname url\endcsname\relax
  \def\url#1{\texttt{#1}}\fi
\expandafter\ifx\csname urlprefix\endcsname\relax\def\urlprefix{URL }\fi
\expandafter\ifx\csname href\endcsname\relax
  \def\href#1#2{#2} \def\path#1{#1}\fi

\bibitem{slater1999essentials}
P.~J. Slater, Essentials of animal behaviour, Cambridge University Press, 1999.

\bibitem{cappe2006inference}
O.~Capp{\'e}, E.~Moulines, T.~Ryd{\'e}n, Inference in hidden Markov models,
  Springer Science \& Business Media, 2006.

\bibitem{puterman2014markov}
M.~L. Puterman, Markov decision processes: discrete stochastic dynamic
  programming, John Wiley \& Sons, 2014.

\bibitem{yin2009hybrid}
G.~G. Yin, C.~Zhu, Hybrid switching diffusions: properties and applications,
  Vol.~63, Springer Science \& Business Media, 2009.

\bibitem{Kass_et_al_99}
R.~E. Kass, L.~Tierney, J.~B. Kadane, The validity of posterior expansions
  based on laplace's method, Bayesian and Likelihood Methods in Statistics and
  Econometrics. In Essays in Honor of George Bernard (eds S. Geisser, J. S.
  Hodges, S. J. Press and A. Zellner) (1990) 473--488.

\bibitem{Wong1989Asymptotic}
R.~Wong, Asymptotic Approximations of Integrals, Academic Press, 1989.

\bibitem{fernandez2015asymptotically}
R.~Fernandez, F.~Manzo, F.~Nardi, E.~Scoppola, Asymptotically exponential
  hitting times and metastability: a pathwise approach without reversibility,
  Electronic Journal of Probability 20 (2015).

\bibitem{bertoncini2008cut}
O.~Bertoncini, J.~Barrera, R.~Fern{\'a}ndez, Cut-off and exit from
  metastability: two sides of the same coin, Comptes Rendus Mathematique
  346~(11-12) (2008) 691--696.

\bibitem{codling2008random}
E.~A. Codling, M.~J. Plank, S.~Benhamou, Random walk models in biology, Journal
  of the Royal society interface 5~(25) (2008) 813--834.

\bibitem{fific2010logical}
M.~Fific, D.~R. Little, R.~M. Nosofsky, Logical-rule models of classification
  response times: a synthesis of mental-architecture, random-walk, and
  decision-bound approaches., Psychological Review 117~(2) (2010) 309.

\bibitem{brown2018ethology}
A.~E. Brown, B.~De~Bivort, Ethology as a physical science, Nature Physics
  14~(7) (2018) 653--657.

\bibitem{erd1946certain}
P.~ERD\&, M.~Kac, On certain limit theorems of the theory of probability,
  Bulletin of the American Mathematical Society 52 (1946) 292--302.

\end{thebibliography}

\end{document}